\documentclass[10pt,twoside]{siamart1116}
\usepackage[english]{babel}
\usepackage{graphicx,epstopdf,epsfig}
\usepackage{amsfonts,epsfig,fancyhdr,graphics, hyperref,amsmath,amssymb}
\usepackage{tikz}

\newcommand{\HE}{Name of Handling Editor}
\newcommand{\DoS}{Month/Day/Year}
\newcommand{\DoA}{Month/Day/Year}
\newcommand{\CA}{Hermie Monterde}

\newcommand{\Names}{Hermie Monterde}
\newcommand{\Title}{Strong cospectrality and twin vertices in weighted graphs}

\newtheorem{example}[theorem]{Example}

\newcommand{\reals}{\mathbb{R}}
\newcommand{\complex}{\mathbb{C}}
 
\renewcommand{\theequation}{\thesection.\arabic{equation}}
\renewtheorem{theorem}{Theorem}[section]

\begin{document}

\bibliographystyle{plain}

%  Leave these commented lines here
%\input{ELAheader-template.tex}
% ELA insert correct page number
\setcounter{page}{1}

\thispagestyle{empty}

%Insert the title of the paper
 \title{\Title\thanks{Received
 by the editors on \DoS.
 Accepted for publication on \DoA. 
 Handling Editor: \HE. Corresponding Author: \CA}}

\author{
Hermie Monterde\thanks{Department of Mathematics,
University of Manitoba, Winnipeg, MB, Canada R3T 2N2,
(monterdh@myumanitoba.ca).}}

\markboth{\Names}{\Title}

\maketitle

\begin{abstract}
We explore algebraic and spectral properties of weighted graphs containing twin vertices that are useful in quantum state transfer. We extend the notion of adjacency strong cospectrality to Hermitian matrices, with focus on the generalized adjacency matrix and the generalized normalized adjacency matrix. We then determine necessary and sufficient conditions such that a pair of twin vertices in a weighted graph exhibits strong cospectrality with respect to the above-mentioned matrices. We also determine when strong cospectrality is preserved under Cartesian and direct products of graphs. Moreover, we generalize known results about equitable and almost equitable partitions, and use these to determine which joins of the form $X\vee H$, where $X$ is either the complete or empty graph, exhibit strong cospectrality.
\end{abstract}


% Sample article for the Electronic Journal of Linear Algebra

%%%%%%%%%%%%%%%%%%%%%%%%%%%%%%%%%%%%%%%

%%%%%%%%%%%%%%%%%%%%%%%%%%%%%%%%%%%%%%%

\section{Introduction} \label{intro-sec}

In the field of quantum state transfer, various types of useful quantum phenomena occur between twin vertices in graphs. For instance, Angeles-Canul et al.\ showed that the apexes of a double cone on an $n$-vertex $k$-regular graph exhibit perfect state transfer with respect to the adjacency matrix for infinitely many $n$ and $k$ \cite{AC2}. For the Laplacian case, Alvir et al.\ proved that the disconnected double cone on an $n$-vertex graph (not necessarily regular) exhibits perfect state transfer between apexes if and only if $n\equiv 2$ (mod 4) \cite{Alvir2016}. The apexes of a double cone are also known to admit fractional revival with respect to the adjacency and Laplacian matrices \cite{Chan2020,Chan2021}. Moreover, it was shown by Kempton et al.\ that for two non-adjacent vertices with the same neighbourhoods, adding a suitable potential (weighted loops) on the vertices yields adjacency perfect state transfer between them \cite{Kempton}. A similar result also holds for double cones as shown by Angeles-Canul et al.: adjacency perfect state transfer can be achieved between the apexes by adding loops with suitable weights on the apexes, or by weighting the edges incident to the apexes \cite{AC1}.

Despite its presence in literature, the role of twin vertices in quantum state transfer has not yet been thoroughly explored. This motivates our work in providing a systematic approach to analyzing the algebraic and spectral properties of graphs containing twin vertices that have significant consequences in quantum state transfer. In this paper, we consider connected weighted undirected graphs with possible loops but no multiple edges. We focus primarily on the two types of matrices associated to graphs, namely, the generalized adjacency matrix, and the generalized normalized adjacency matrix, with emphasis on the generalized adjacency matrix in the last two sections. These matrices are generalizations of the four most common types of matrices associated to weighted graphs, namely, the adjacency matrix and graph Laplacians (Laplacian matrix, signless Laplacian matrix, and normalized Laplacian matrix).

This paper is organized as follows. In Section \ref{spectralproperties}, we mainly extend the notions of cospectral, parallel and strongly cospectral vertices to columns of Hermitian matrices with special attention to generalized adjacency and generalized normalized adjacency matrices. Our extension preserves a well-known fact relative to the adjacency matrix of a graph: two columns of a Hermitian matrix are strongly cospectral if and only if they are cospectral and parallel. We provide necessary and sufficient conditions for two columns of a Hermitian matrix to be cospectral and parallel. We also look at the spectral and algebraic properties of these matrices whenever the underlying graph contains twin vertices. In particular, we show that a weighted graph contains two twin vertices if and only if there is an involution that switches them and fixes all other vertices. An important consequence of this fact is that twin vertices are cospectral, and thus, twin vertices are strongly cospectral if and only if they are parallel. In Section \ref{strongcospectrality}, we present spectral properties of strongly cospectral vertices. We prove our main result, which is a characterization of twin vertices that are strongly cospectral and its implications on the columns of the orthogonal projection matrices in the spectral decompositions of the generalized adjacency and generalized normalized adjacency matrices. A characterization of strong cospectrality between columns of Hermitian matrices is also given. Moreover, we show that strong cospectrality in simple unweighted graphs is closed under complementation. In Section \ref{prod}, we determine under which conditions is strong cospectrality preserved under Cartesian and direct products. Section \ref{eqtp} is dedicated to equitable and almost equitable partitions. In particular, we establish that strong cospectrality between two vertices which are singleton cells in either an equitable or almost equitable partition is preserved in the quotient graph, and vice versa. We also extend a result of Bachman et al.\ in \cite[Theorem 2]{Bachman2011}. Finally, in Section \ref{joins}, we identify which joins of the form $X\vee H$, where $X$ is either a weighted complete graph or a weighted empty graph and $H$ is an $m$-vertex graph, exhibit strong cospectrality between the vertices of $X$. As it turns out, double cones, which are joins of this form, provide plenty of examples of strongly cospectral vertices.

Some results in this paper as well as an extensive treatment of twin vertices in quantum state transfer can be found in the work of Monterde \cite[MSc. Thesis]{Monterde}. For the basics of graph theory and matrix
theory, we refer the reader to Godsil and Royle \cite{Godsil:AlgebraicGraph}, and Horn and Johnson \cite{Horn:MatrixAnalysis,horn94}. For more background on quantum state transfer, see Godsil \cite{Godsil2012a}, and Coutinho and Godsil \cite{Coutinho2021}.

Throughout this paper, $X$ denotes a connected weighted undirected graph with possible loops but no multiple edges. The edges of $X$ have non-zero real weights, and we denote the vertex set of $X$ by $V(X)$. If $X$ is disconnected, then we apply our results to the components of $X$.  We say that $X$ is \textit{simple} if $X$ has no loops, $X$ is \textit{positively weighted} (resp., \textit{negatively weighted}) if all edge weights in $X$ are positive (respectively, negative), and $X$ is \textit{unweighted} if all edges of $X$ have weight one. If $X$ is simple and unweighted, then we denote the complement of $X$ by $\overline{X}$. For $u\in V(X)$, we denote the subgraph of $X$ induced by $V(X)\backslash\{u\}$ as $X\backslash u$, the set of neighbours of $u$ in $X$ as $N_X(u)$, and the characteristic vector of $u$ as $\textbf{e}_u$, which is a vector with a $1$ on the entry indexed by $u$ and $0$'s elsewhere. The all ones vector of order $n$, the zero vector of order $n$, the $n\times n$ all ones matrix, and the $n\times n$ identity matrix are denoted by $\textbf{1}_n$, $\textbf{0}_n$, $J_n$ and $I_n$, respectively. If the context is clear, then we write $\textbf{1}$, $\textbf{0}$, $J$ and $I$, respectively. If $Y$ is another graph, then we use $X\cong Y$ to denote the fact that $X$ and $Y$ are isomorphic, and adopt the notation $X\vee Y$ for the join of $X$ and $Y$. We also represent the conjugate transpose and the transpose of a matrix $M$ respectively by $M^*$ and $M^T$. Moreover, we write the characteristic polynomial of a square matrix $M$ in the variable $t$ as $\phi(M,t)$, and the spectrum of $M$ as $\sigma(M)$. Lastly, we denote the simple unweighted empty, cycle, complete and path graphs on $n$ vertices as $O_n$, $C_n$, $K_n$ and $P_n$, and the simple unweighted complete graph minus an edge by $K_n\backslash e$.

Two vertices $u$ and $v$ of $X$ are \textit{twins} if the following conditions hold.
\begin{enumerate}
\item $N_X(u)\backslash \{u,v\}=N_X(v)\backslash \{u,v\}$.
\item The edges $(u,w)$ and $(v,w)$ have the same weight for each $w\in N_X(u)\backslash \{u,v\}$.
\item If there are loops on $u$ and $v$, then they have the same weight.
\end{enumerate}
In addition, if $u$ and $v$ are adjacent, then $u$ and $v$ are \textit{true twins}. Otherwise, $u$ and $v$ are \textit{false twins}. Note that our definition generalizes the definition of twin vertices from simple unweighted graphs to weighted graphs with possible loops. In literature, false and true twins in simple unweighted graphs are also known as duplicates and co-duplicates, respectively \cite{Briffa2020}. Observe that if either $u$ and $v$ are false twins or if $u$ and $v$ are true twins with loops, then $N_X(u)=N_X(v)$, i.e., $u$ and $v$ have the same open and closed neighbourhoods, respectively. However, this is not true if either $u$ and $v$ are true twins without loops or if $u$ and $v$ are false twins with loops. We also add that if $u$ and $v$ are twins, then $\operatorname{deg}(u)=\operatorname{deg}(v)$, and if $X$ is simple and unweighted, then $u$ and $v$ are twins if and only if $N_X(u)\backslash \{v\}=N_X(v)\backslash \{u\}$.

Let $\omega,\eta\in\reals$. A subset $T=T(\omega,\eta)$ of $V(X)$ is a \textit{set of twins} in $X$ if any two vertices in $T$ are pairwise twins, where each vertex in $T$ has a loop of weight $\omega$, and the loops are absent if $\omega=0$, and every pair of vertices in $T$ are connected by an edge with weight $\eta$, and every pair of vertices in $T$ are not adjacent whenever $\eta=0$. Note that if $T$ is a set of twins in $X$, then either every pair of distinct vertices in $T$ are true twins, in which case $\eta\neq 0$, or every pair of distinct vertices in $T$ are false twins, in which case $\eta=0$. In particular, if $X$ is a simple unweighted graph, then $\omega=0$ and $\eta\in\{0,1\}$.

\begin{example}
In $K_n$, any pair of vertices are true twins. Meanwhile, in $K_n\backslash e$, the two non-adjacent vertices form a set of false twins, while the rest of the $n-2$ adjacent vertices form a set of true twins. 
\end{example}

The adjacency, Laplacian, signless Laplacian, and normalized Laplacian matrices of $X$, denoted $A(X)$, $L(X)$, $Q(X)$ and $\mathcal{L}(X)$, are fundamental in quantum state transfer because they serve as Hamiltonians for nearest-neighbour interactions of qubits in a quantum spin system represented by $X$. These Hamiltonians are then used to define continuous-time quantum walks on $X$ \cite{Alvir2016,Godsil2012a,Kendon2003,Kendon2011}. In general, a Hamiltonian is a Hermitian matrix associated to a graph that is known to have the property that its entry indexed by two vertices is zero if and only if there is no edge between them. For this reason, we consider generalizations of these four matrices that respect the adjacencies of the vertices in $X$, regardless of edge weights. To do this, we recall that the adjacency matrix $A(X)$ of $X$ is the matrix such that
\begin{equation*}
(A(X))_{j,\ell}=
\begin{cases}
 \omega_{j,\ell}, &\text{if $j$ is adjacent to $\ell$}\\
 0, &\text{otherwise},
\end{cases}
\end{equation*}
where $0\neq \omega_{j,\ell}\in\mathbb{R}$ is the weight of the edge $(j,\ell)$. The degree matrix $D(X)$ of $X$ is the diagonal matrix of vertex degrees of $X$, where $\operatorname{deg}(u)=2\omega_{u,u}+\sum_{j\neq u}\omega_{u,j}$ for each $u\in V(X)$. Note that $\operatorname{deg}(u)=0$ is possible without $u$ being an isolated vertex, and it can also happen that the vertex degrees in the graph can be all positive, or all negative, even if the edge weights themselves may have mixed signs. If $\operatorname{deg}(u)\neq 0$ for each $u\in V(X)$, then we define $D(X)^{-\frac{1}{2}}$ as the diagonal matrix such that $(D(X))_{u,u}=1/\sqrt{\operatorname{deg}(u)}$ if $\operatorname{deg}(u)>0$, while $(D(X))_{u,u}=-i/\sqrt{\operatorname{deg}(u)}$ if $\operatorname{deg}(u)<0$, where $i^2=-1$.

\begin{definition}
Let $\alpha,\beta,\gamma\in\reals$ with $\gamma\neq 0$.
\begin{enumerate}
\item A \textit{generalized adjacency matrix} $\textbf{A}(X)$ of $X$ is a matrix of the form
\begin{equation}
\label{gam}
\textbf{A}(X)=\alpha I+\beta D(X)+\gamma A(X).
\end{equation}
\item If either $\operatorname{deg}(u)>0$ for each $u\in V(X)$ or $\operatorname{deg}(u)<0$ for each $u\in V(X)$, then a \textit{generalized normalized adjacency matrix} $\mathcal{A}(X)$ of $X$ is a matrix of the form
\begin{equation}
\label{gnam}
\mathcal{A}(X)=\alpha I+\gamma D(X)^{-\frac{1}{2}}AD(X)^{-\frac{1}{2}}.
\end{equation} 
\end{enumerate}
\end{definition}

Generalized adjacency matrices were introduced by van Dam and Haemers in \cite{VanDam} as linear combinations of $I$, $J$, $D(X)$ and $A(X)$, with the goal of characterizing simple unweighted graphs that are determined by their spectrum. But since we deal with weighted graphs, we omit $J$ in our definition. If $\alpha=0$, $0\leq \beta\leq 1$, and $\gamma=1-\beta$, then our definition of a generalized adjacency matrix coincides with the matrix $A_{\beta}$ proposed by Nikiforov to merge the study of adjacency and signless Laplacian matrices \cite{Nikiforov}. We also note that in (\ref{gnam}), $X$ need be neither positively nor negatively weighted, and the connectedness of $X$ implies that the diagonal entries of $D(X)^{-\frac{1}{2}}$ are either all positive or all purely imaginary, and so $\mathcal{A}(X)$, like $\textbf{A}(X)$, is Hermitian.

We reserve the symbols $\alpha$, $\beta$ and $\gamma$ to denote the parameters of $\textbf{A}(X)$ and $\mathcal{A}(X)$ in (\ref{gam}) and (\ref{gnam}). Unless otherwise specified, we use $M(X)$ to denote $\textbf{A}(X)$ or $\mathcal{A}(X)$. If the context is clear, then we simply write our matrices as $M$, $A$, $L$, $Q$, $\mathcal{L}$, $\textbf{A}$, $\mathcal{A}$, and $D$ for brevity. Note that if $\alpha=\beta=0$ and $\gamma=1$, then $\textbf{A}=A$, while if $\alpha=0$, $\beta=1$ and $\gamma=\mp 1$, then $\textbf{A}=L$ or $Q$. Also, if $\alpha=1$ and $\gamma=-1$, then $\mathcal{A}=\mathcal{L}$. That is, $A$, $L$, and $Q$ are generalized adjacency matrices, while $\mathcal{L}$ is a generalized normalized adjacency matrix. It is also worth mentioning that $X\cong Y$ if and only if there exists a permutation matrix $P$ such that $M(X)=P^TM(Y)P$ whenever $M=\textbf{A}$ or $\mathcal{A}$. Lastly, if $X$ is a weighted $k$-regular graph, then $\alpha+(\beta+\gamma) k$ and $\alpha+\gamma$ are simple eigenvalues of $\textbf{A}$ and $\mathcal{A}$, resp., both with eigenvector $\textbf{1}$.

%%%%%%%%%%%%%%%%%%%%%%%%%%%%%%%%%%%%%%%

\section{Spectral and algebraic properties}\label{spectralproperties}

Let $H$ be an $m\times m$ Hermitian matrix. Then $\sigma(H)\subseteq \reals$ and $H$ admits a spectral decomposition
\begin{equation*}
H=\sum_{j}\lambda_jE_j
\end{equation*} 
where the $\lambda_j$'s are the distinct eigenvalues of $H$ and each $E_j$ is the orthogonal projection matrix onto the eigenspace associated with $\lambda_j$. Each $E_j$ can be uniquely determined by taking any orthonormal set of eigenvectors $\{\textbf{v}_1,\ldots\textbf{,v}_{\ell}\}$ for $H$ corresponding to the eigenvalue $\lambda_j$ and setting $E_j=\sum_{k=1}^{\ell}\textbf{v}_k\textbf{v}_k^*$.
The $E_j$'s are Hermitian, idempotent, and pairwise multiplicatively orthogonal. They also commute with $A$, and they sum to identity. In addition, if $f$ is a function that is analytic on its domain $D$ and $\lambda_j\in D$ for each $j$, then
\begin{equation}
\label{anal}
f(H)=\sum_{j}f(\lambda_j)E_j.
\end{equation}
For more about this fact, see \cite[Definition 6.2.4]{horn94}. In particular, if $p(x)$ is a polynomial satisfying $p(\lambda_j)=0$ for $j\neq \ell$ and $p(\lambda_{\ell})=1$, then $p(H)=\sum_jp(\lambda_j)E_j=E_{\ell}$. That is, each $E_j$ is a polynomial in $H$.

Let $u,v\in\{1,\ldots,m\}$ be columns of $H$. The \textit{eigenvalue support} $\sigma_u(H)$ of $u$ with respect to $H$ is the set $\sigma_u(H)=\{\lambda_j:E_j\textbf{e}_u\neq 0\}$. With respect to $H$, we say that $u$ and $v$ are
\begin{enumerate}
\item \textit{cospectral} if $(E_j)_{u,u}=(E_j)_{v,v}$ for each $j$,
\item \textit{parallel} if $E_j\textbf{e}_u$ and $E_j\textbf{e}_v$ are parallel vectors for each $j$, i.e., for each $j$, there exists $c\in\complex$ such that
\begin{equation}
\label{parallel}
E_j\textbf{e}_u=cE_j\textbf{e}_v,
\end{equation}
\item \textit{strongly cospectral} if for each $j$, (\ref{parallel}) holds for some $c\in\complex$ with $|c|=1$. 
\end{enumerate}
Note that cospectrality, parallelism and strong cospectrality with respect to $H$ are equivalence relations. If $H$ is real symmetric, then $u$ and $v$ are parallel with respect to $H$ if the constant $c$ in (\ref{parallel}) is real, and strongly cospectral with respect to $H$ if $E_j\textbf{e}_u=\pm E_j\textbf{e}_v$ for each $j$, in which case we define the sets
\begin{equation*}
\sigma_{uv}^+(H)=\{\lambda_j:E_j\textbf{e}_u=E_j\textbf{e}_v\}\ \text{and}\ \sigma_{uv}^-(H)=\{\lambda_j:E_j\textbf{e}_u=-E_j\textbf{e}_v\}.
\end{equation*}

Cospectrality is a concept that dates back to Schwenk \cite{Schwenk}. In contrast, strong cospectrality is a more recent concept that appeared in \cite[Lemma 11.1]{Godsil2012a}, and possibly earlier, but was only defined explicitly in \cite{God,Fan2013}. Parallelism was introduced by Godsil and Smith \cite{Godsil2017}. In this section, we extend these notions to any two columns of a Hermitian matrix, and generalize known results about adjacency and Laplacian strong cospectrality. In later chapters, we mostly deal with the case $H=M(X)$, where $M=\textbf{A}$ or $M=\mathcal{A}$. In Section \ref{eqtp}, we also handle the case when $H=\mathcal{M}(X/\pi)$, where $\mathcal{M}$ is a matrix associated to the quotient graph $X/\pi$. In particular, if $H=A(X)$, we use the terms adjacency strongly cospectral, adjacency cospectral, and adjacency parallel interchangeably with strongly cospectral with respect to $A(X)$, cospectral with respect to $A(X)$, and parallel with respect to $A(X)$, respectively. Similarly, for $H=L(X),\ Q(X)$, or $\mathcal{L}(X)$.

If $u$ and $v$ are cospectral with respect to $H$, then $\sigma_u(H)=\sigma_v(H)$, while if $u$ and $v$ are strongly cospectral with respect to a real symmetric $H$, then $\sigma_u(H)=\sigma_{uv}^+(H)\cup \sigma_{uv}^-(H)$. Strong cospectrality is relevant in quantum state transfer because it is a condition necessary for some quantum phenomena to occur between two vertices in a graph, such as pretty good state transfer \cite[Lemma 13.1]{Godsil2012a}, as well as Laplacian fractional revival \cite[Theorem 7.6]{Chan2020}. For more about adjacency strong cospectrality in unweighted graphs, see \cite{Godsil2017}.
 
We now provide a characterization of cospectrality between columns of Hermitian matrices which is based closely on \cite[Theorem 3.1]{Godsil2017}. For $S\subseteq\{1,\ldots,m\}$, we denote by $H[S]$ the submatrix of $H$ obtained by deleting its row and column indexed by elements in $S$ and $\phi_S(H,t):=\text{det}((tI-H)[S])$. If $S=\{u\}$, then we write $H[S]$ and $\phi_S(H,t)$ as $H[u]$ and $\phi_u(H,t)$, respectively.

\begin{theorem}
\label{cosp}
Let $H$ be an $m\times m$ Hermitian matrix. Then the elements of $\sigma_u(H)$ are the poles of $\phi_u(H,t)/\phi(H,t)$. Moreover, the following are equivalent.
\begin{enumerate}
\item $u$ and $v$ are cospectral with respect to $H$.
\item $(E_j)_{u,u}=(E_j)_{v,v}$ for each orthogonal projection matrix $E_j$ of $H$.
\item $\phi_u(H,t)=\phi_v(H,t)$.
\item $(f(H))_{u,u}=(f(H))_{v,v}$ for any function $f$ that is analytic on its domain $D$ and each eigenvalue of $H$ is contained in $D$.
\end{enumerate}
\end{theorem}

The term \textit{cospectral} vertices comes from the fact that if $H=A(X)$ and $u$ and $v$ are vertices of $X$, then $\phi_u(A(X),t)=\phi_v(A(X),t)$ is equivalent to $\phi(A(X\backslash u),t)=\phi(A(X\backslash v),t)$ so that $u$ and $v$ are adjacency cospectral if and only if the vertex deleted graphs $X\backslash u$ and $X\backslash v$ have the same adjacency spectrum. However, we point out that this need not apply whenever $H=\textbf{A}(X)$ and $\beta\neq 0$, or $H=\mathcal{A}(X)$.

For $u,v\in \{1,\ldots,m\}$, the \textit{walk matrix} of $H$ relative to $u$ is
\begin{equation*}
W_u(H)=\left[\textbf{e}_u\quad H\textbf{e}_u\quad \ldots\quad H^{n-1}\textbf{e}_u\right].
\end{equation*}
In particular, we call $W_u(H)-W_v(H)$ and $W_u(H)+W_v(H)$ as the weighted walk matrices relative to $\textbf{e}_u-\textbf{e}_v$ and $\textbf{e}_u+\textbf{e}_v$, respectively. The column space of $W_u(H)$ is the $H$-invariant subspace of $\complex^n$ generated by $\textbf{e}_u$. It is in fact the $\complex[H]$-module generated by $\textbf{e}_u$, where $\complex[H]$ denotes the ring of polynomials in $H$ with complex coefficients. In particular, if $H$ is real symmetric, then the column space of $W_u(H)$ is the $H$-invariant subspace of $\mathbb{R}^n$ generated by $\textbf{e}_u$, and is an $\mathbb{R}[H]$-module generated by $\textbf{e}_u$. For brevity, we simply refer to the column space of $W_u(H)$ as the $H$-module generated by $\textbf{e}_u$. We say that the $H$-modules generated by $\textbf{e}_u-\textbf{e}_v$ and $\textbf{e}_u+\textbf{e}_v$ are orthogonal if and only if any two columns of $W_u(H)-W_v(H)$ and $W_u(H)+W_v(H)$ are orthogonal, i.e., $\langle H^j(\textbf{e}_u-\textbf{e}_v),H^{\ell}(\textbf{e}_u+\textbf{e}_v)\rangle=0$. Equivalently, $(E_j)_{u,u}+2\operatorname{Im}(E_j)_{u,v}-(E_j)_{v,v}=0$ for each $j$. In the context of simple unweighted graphs where $H=A(X)$, we call $W_u(A(X))$ as the \textit{walk matrix} relative to $\textbf{e}_u$, and we call the $A$-module generated by $\textbf{e}_u$ as a \textit{walk module}. For more about walk matrices and walk modules in simple unweighted graphs, see \cite{Coutinho2021}.

The following is another characterization of cospectral columns in real symmetric matrices, which can be considered as an extension of some results in Theorem 3.1 and Lemma 11.1 in \cite{Godsil2017}.

\begin{theorem}
\label{cosp}
Let $H$ be an $m\times m$ Hermitian matrix. The $H$-modules generated by $\textbf{e}_u-\textbf{e}_v$ and $\textbf{e}_u+\textbf{e}_v$ are orthogonal subspaces of $\mathbb{C}^m$ if and only if $(E_j)_{u,u}+2\operatorname{Im}(E_j)_{u,v}-(E_j)_{v,v}=0$ for each $j$. If the $H$-modules generated by $\textbf{e}_u-\textbf{e}_v$ and $\textbf{e}_u+\textbf{e}_v$ are orthogonal subspaces of $\mathbb{C}^m$, then $u$ and $v$ are cospectral with respect to $H$ if and only if $(E_j)_{u,v}\in\mathbb{R}$  for each $j$. If we add that $H$ is real symmetric, then the following are equivalent.
\begin{enumerate}
\item $u$ and $v$ are cospectral with respect to $H$.
\item The $H$-modules generated by $\textbf{e}_u-\textbf{e}_v$ and $\textbf{e}_u+\textbf{e}_v$ are orthogonal subspaces of $\mathbb{R}^m$.
\item \label{6} There is a real symmetric matrix $R$ such that $R^2=I$, $R\textbf{e}_u=\textbf{e}_v$, and $RH=HR$.
\end{enumerate}
\end{theorem}

We also mention a useful sufficient condition for two columns of a Hermitian matrix to be cospectral.

\begin{corollary}
\label{cospaut}
Let $H$ be a Hermitian matrix. If $P$ is a permutation matrix such that $H=P^THP$ and $P\textbf{e}_u=\textbf{e}_v$, then $u$ and $v$ cospectral with respect to $H$.
\end{corollary}

If there is an automorphism of $X$ that sends $u$ to $v$, then they are automatically cospectral. Now, if $\beta=0$, then $\phi_u(\textbf{A}(X),t)=\phi(\textbf{A}(X\backslash u),t)$, and so even if no automorphism sends $u$ to $v$, $X\backslash u\cong X\backslash v$ yields adjacency cospectrality between $u$ and $v$ (see Schwenk's tree in \cite{Godsil2017} for example). However, the converse is not true. Indeed, if $X\cong Y(1,-1)$ in Figure \ref{plantsie}, then vertices $u$ and $v$ marked blue are adjacency cospectral but $X\backslash u\not\cong X\backslash v$. While vertices $u$ and $v$ can be cospectral despite the absence of an automorphism that sends one to the other, Theorem \ref{cosp}(\ref{6}) guarantees a matrix $R$ that behaves like an involution that switches $u$ and $v$. However, as pointed out in \cite{Godsil2017}, $R$ need not be related to any automorphism of $X$.

Now, let $M=\textbf{A}$ with $\beta\neq 0$ or $M=\mathcal{A}$. Then $\phi_u(M(X),t)$ and $\phi(M(X\backslash u),t)$ need not be equal, i.e., $X\backslash u\cong X\backslash v$ does not guarantee cospectrality with respect to $M$. Indeed, if $X\cong T$ in Figure \ref{plantsie}, then $X\backslash u\cong X\backslash v$, but $u$ and $v$ are not cospectral with respect to $M=\textbf{A}$ whenever $\beta\neq 0$, as well as $M=\mathcal{A}$. As a result, if $X$ is fixed and $u$ and $v$ are cospectral with respect to $\textbf{A}$ for some $\alpha,\beta$ and $\gamma$, then $u$ and $v$ need not be cospectral with respect to $\textbf{A}$ when at least one of $\beta$ and $\gamma$ is changed. Nonetheless, since $M$ and $aI+bM$ have the same set of orthogonal projection matrices, cospectrality and parallelism are preserved with respect to $\textbf{A}$ if $\alpha$ is changed, and with respect to $\mathcal{A}$ if both $\alpha$ and $\gamma$ are changed.

Next, we state the following variations of Corollaries 6.3 and 6.4 in \cite{Godsil2017}, respectively.

\begin{theorem}
\label{paryup1}
Let $H$ be an $m\times m$ Hermitian matrix and $u,v\in\{1,\ldots,m\}$. If $u$ and $v$ are parallel with respect to $H$, $\sigma_u(H)=\sigma_v(H)$, and $R$ is a matrix such that $RH=HR$ and $R\textbf{e}_u=\textbf{e}_u$, then $R\textbf{e}_v=\textbf{e}_v$. 
\end{theorem}

\begin{corollary}
\label{paryup}
If $u$ and $v$ are parallel with respect to $M$ and $\sigma_u(H)=\sigma_v(H)$, then any automorphism of $X$ that fixes $u$ must fix $v$.
\end{corollary}

Thus, for two strongly cospectral vertices, any automorphism of $X$ that fixes one must fix the other.

Let $H$ be an $m\times m$ Hermitian matrix and $S\subseteq\{1,\ldots,m\}$. Denote by $H_S$ the submatrix of $H$ whose entries are indexed by elements in $S$. Following the proof of \cite[Theorem 4.5.1]{Coutinho2021}, we get that
\begin{equation}
\label{yo}
\text{det}\left((tI-H)^{-1}\right)_{S}=\frac{\phi_S(H,t)}{\phi(H,t)}.
\end{equation}

\begin{figure}[h!]
\begin{center}
		\begin{tikzpicture}
		\tikzset{enclosed/.style={draw, circle, inner sep=0pt, minimum size=.25cm}}
 		
 		\node[enclosed,fill=cyan] (y_1) at (1.2,1.15) {};
		\node[enclosed] (y_2) at (0.7,-0.3) {};
		\node[enclosed,fill=cyan] (y_3) at (2.3,1.15) {};
		\node[enclosed] (y_4) at (2.8,-0.3) {};
		
		\draw (y_1) -- node[left] {$a$} (y_2);
		\draw (y_1) -- node[below,right] {$a$} (y_4);
		\draw (y_2) -- node[below,left] {$b$} (y_3); 	
 		\draw (y_3) -- node[right] {$b$} (y_4); 
 		\draw (y_3) -- node[above] {$b$} (y_1);
		\draw (y_3) to[in=90,out=0,loop,looseness=10] node[above] {$a$} (y_3);
		\draw (y_1) to[in=90,out=180,loop,looseness=10] node[above] {$a$} (y_1);
		
		\node[enclosed] (x_1) at (4.9,1.1) {};
		\node[enclosed] (x_2) at (4.9,0) {};
		\node[enclosed] (x_3) at (6,0) {};
		\node[enclosed,fill=cyan,label={above, yshift=0cm: $u$}] (x_4) at (7.1,0) {};
		\node[enclosed] (x_5) at (8.2,0) {};
		\node[enclosed] (x_6) at (9.3,0) {};
		\node[enclosed,fill=cyan,label={above, yshift=0cm: $v$}] (x_7) at (10.4,0) {};
		\node[enclosed] (x_8) at (11.5,0) {};
		\node[enclosed] (x_9) at (8.2,1.1) {};
		\node[enclosed] (x_10) at (11.5,1.1) {};
		\node[enclosed] (x_11) at (12.6,0) {};
		
		\draw (x_1) -- (x_2);
		\draw (x_2) -- (x_3);
 		\draw (x_3) -- (x_4);
 		\draw (x_4) -- (x_5);
 		\draw (x_5) -- (x_6);
 		\draw (x_6) -- (x_7);
 		\draw (x_5) -- (x_9);
 		\draw (x_7) -- (x_8);
 		\draw (x_10) -- (x_8);
 		\draw (x_11) -- (x_8);
		\end{tikzpicture}
	\end{center}
	\caption{The weighted graph $Y(a,b)$ (left), and an unweighted tree $T$ with no automorphism mapping $u$ to $v$ but $X\backslash u\cong X\backslash v$ (right)}\label{plantsie}
\end{figure}
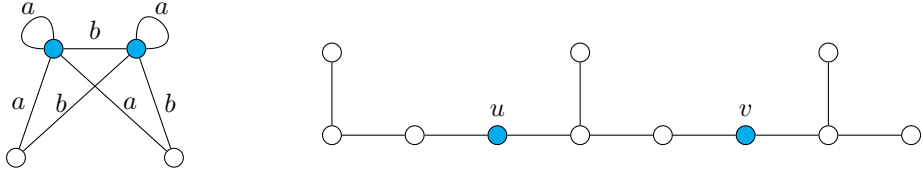

Our next result is a characterization of parallelism. We follow the work of Godsil and Smith \cite{Godsil2017}.

\begin{theorem}
\label{parchar}
Let $H$ be an $m\times m$ Hermitian matrix. The following hold.
\begin{enumerate}
\item \label{p1} $u$ and $v$ are parallel with respect to $H$ if and only if the poles of $\frac{\phi_S(H,t)}{\phi(H,t)}$ are simple, where $S=\{u,v\}$.
\item \label{p2} $u$ and $v$ are parallel with respect to $H$ and $\sigma_u(H)=\sigma_v(H)$ if and only if the $H$-modules generated by $\textbf{e}_u$ and $\textbf{e}_v$ are equal.
\end{enumerate}
\end{theorem}

\begin{proof}
Note that 1 holds as a straightforward consequence of \cite[Lemma 2.3]{CG}. Now, let $H=\sum_{j=1}^r\lambda_j E_j$ be a spectral decomposition of $H$. For each $w$ in the column space of $W_u(H)$,
\begin{equation}
\label{11}
w=\sum_{\ell=1}^{n}a_{\ell}H^{\ell}\textbf{e}_u=\sum_{\ell=1}^n a_{\ell}\left(\sum_{j=1}^r\lambda_j^{\ell}E_j\right)\textbf{e}_u=\sum_{j=1}^r \left(\sum_{\ell=1}^n a_{\ell}\lambda_j^{\ell}\right)E_j\textbf{e}_u.
\end{equation}
Thus, $\{E_1\textbf{e}_u,\ldots,E_r\textbf{e}_u\}$ form an orthogonal basis for the $H$-module generated by $e_u$. If $u$ and $v$ are parallel with respect to $H$ and $\sigma_u(H)=\sigma_v(H)$, then the $H$-modules generated by $\textbf{e}_u$ and $\textbf{e}_v$ are equal. Conversely, if the $H$-modules generated by $\textbf{e}_u$ and $\textbf{e}_v$ are equal, then $\{E_1\textbf{e}_v,\ldots,E_r\textbf{e}_v\}$ is an orthogonal basis for the $H$-module generated by $e_u$. By (\ref{11}), $w=\sum_{j=1}^rb_jE_j\textbf{e}_u=\sum_{j=1}^rc_jE_j\textbf{e}_v$, and so $b_kE_k\textbf{e}_u=c_kE_k\textbf{e}_v$ for each $k\in\{1,\ldots,r\}$. That is, $u$ and $v$ are parallel with respect to $H$ and $\sigma_u(H)=\sigma_v(H)$. Hence, \ref{p2} holds.
\end{proof}

Let us now explore the connection between the degrees of cospectral vertices.

\begin{proposition}
\label{deg}
Let $u$ and $v$ be cospectral vertices in $X$ with respect to $M$.
\begin{enumerate}
\item If $M=\textbf{A}$, then $\beta\hspace{0.01in}\operatorname{deg}(u)+\gamma (A)_{u,u}=\beta\hspace{0.01in}\operatorname{deg}(v)+\gamma (A)_{v,v}$.
\item Let $M=\mathcal{A}$, Then $\operatorname{deg}(v) (A)_{u,u}=\operatorname{deg}(u) (A)_{v,v}$, and
\begin{equation}
\label{nl}
\sum_{j\in N_X(u)}((A)_{j,u})^2\frac{\operatorname{deg}(v)}{\operatorname{deg}(j)}=\sum_{j\in N_X(v)}((A)_{j,v})^2\frac{\operatorname{deg}(u)}{\operatorname{deg}(j)}.
\end{equation}
If we add that $X$ is unweighted, then
\begin{equation}
\label{unwnl}
\sum_{j\in N_X(u)}\frac{\operatorname{deg}(v)}{\operatorname{deg}(j)}=\sum_{j\in N_X(v)}\frac{\operatorname{deg}(u)}{\operatorname{deg}(j)},
\end{equation}
\end{enumerate} 
\end{proposition}

\begin{proof}
Consider the analytic function $f(x)=x^k$ for any integer $k\geq 0$. Since $u$ and $v$ are cospectral with respect to $M$, Proposition \ref{cosp}(4) implies that $(M^k)_{u,u}=(M^k)_{v,v}$.
If $M=\textbf{A}$, then taking $k=1$ in the preceding equation proves (1), while if $M=\mathcal{A}$, then taking $k=1,2$ proves (2).
\end{proof}

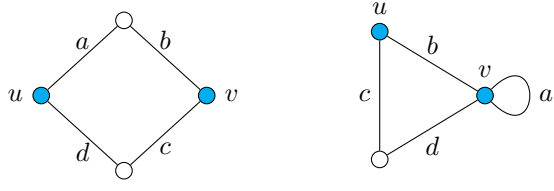
\begin{figure}[h!]
	\begin{center}
		\begin{tikzpicture}
		\tikzset{enclosed/.style={draw, circle, inner sep=0pt, minimum size=.22cm}}
 		
 		\node[enclosed,fill=cyan,label={left, yshift=0cm: $u$}] (y_1) at (0,0) {};
		\node[enclosed] (y_2) at (1.1,1) {};
		\node[enclosed] (y_3) at (1.1,-1) {};
		\node[enclosed,fill=cyan,label={right, yshift=0cm: $v$}] (y_4) at (2.2,0) {};
		
		\draw (y_1) -- node[above] {$a$} (y_2);
		\draw (y_2) -- node[above] {$b$} (y_4);
		\draw (y_3) -- node[below] {$c$} (y_4); 	
 		\draw (y_3) -- node[below] {$d$} (y_1); 
 		
 		\node[enclosed,fill=cyan,label={above, yshift=0cm: $u$}] (x_1) at (4.5,0.85) {};
		\node[enclosed] (x_3) at (4.5,-0.85) {};
		\node[enclosed,fill=cyan,label={above, yshift=0cm: $v$}] (x_4) at (5.9,0) {};
		
		\draw (x_4) to[in=45,out=-45,loop,looseness=15] node[right] {$a$} (x_4);
		\draw (x_3) -- node[below] {$d$} (x_4); 	
 		\draw (x_3) -- node[left] {$c$} (x_1); 
 		\draw (x_4) -- node[above] {$b$} (x_1); 
		\end{tikzpicture}
	\end{center}
	\caption{Weighted cycles $C_4(a,b,c,d)$ (left) and $C_3(a,b,c,d)$ (right)}\label{cycle}
\end{figure}

Let $u$ and $v$ be cospectral with respect to $M$. First, suppose $M=\textbf{A}$. If $\beta=0$, then the loops on $u$ and $v$ have equal weights. In particular, if $X$ is unweighted, then $((A)_{u,u})^2=\operatorname{deg}(u)-(A)_{u,u}$, and so $\operatorname{deg}(u)=\operatorname{deg}(v)$. But if $\beta\neq 0$, then $\operatorname{deg}(u)=\operatorname{deg}(v)$ if and only if the loops on $u$ and $v$ have equal weights. In particular, if $X$ has no loops, then $\operatorname{deg}(u)=\operatorname{deg}(v)$. Consequently, adjacency, Laplacian and signless Laplacian cospectral vertices in simple unweighted graphs have equal degrees. However, we note that it is possible for adjacency cospectral vertices in weighted graphs to have unequal degrees. For example, $X\cong C_4(-1,1,1,-1)$ in Figure \ref{cycle} has adjacency cospectral vertices $u$ and $v$ but $\operatorname{deg}(u)=-2$ while $\operatorname{deg}(v)=2$.

Now, let $M=\mathcal{A}$. If $\operatorname{deg}(u)=\operatorname{deg}(v)$, then the loops on $u$ and $v$ have equal weights. Conversely, if the loops on $u$ and $v$ have equal nonzero weights, then $\operatorname{deg}(u)=\operatorname{deg}(v)$. Moreover, one can use (\ref{unwnl}) to show that two given vertices are not cospectral with respect to $\mathcal{A}$. For instance, since a leaf and a vertex of degree three of $T$ in Figure \ref{plantsie} do not satisfy (\ref{unwnl}), they are not cospectral with respect to $\mathcal{A}$. Meanwhile, vertices $u$ and $v$ of $T$ in Figure \ref{plantsie} satisfy (\ref{unwnl}), but they are not cospectral with respect to $\mathcal{A}$.

We now give a lower bound on the sizes of eigenvalue supports of arbitrary vertices in connected weighted graphs. We denote the largest and smallest eigenvalues of $M$ by $\lambda_{\max}$ and $\lambda_{\min}$, respectively. 

\begin{proposition}
\label{esupp}
Let $H$ be an $m\times m$ irreducible Hermitian matrix. Then $|\sigma_u(H)|\geq 2$ for all $u\in\{1,\ldots,m\}$. In particular, $|\sigma_u(M)|\geq 2$, and the following hold.
\begin{enumerate}
\item \label{posneg} Let $\beta\geq 0$ and $\gamma>0$. If $X$ is positively weighted (resp., negatively weighted), then $\lambda_{\max}\in\sigma_u(M)$ (resp., $\lambda_{\min}\in\sigma_u(M)$). If $u$ and $v$ are strongly cospectral with respect to $M$, then $\lambda_{\max}\in \sigma_{uv}^+(M)$ (resp., $\lambda_{\min}\in \sigma_{uv}^+(M)$).
\item \label{yehh}
\begin{enumerate}
\item If $X$ is weighted $k$-regular, then $\alpha+(\beta+\gamma) k\in\sigma_u(\textbf{A})$ and $\alpha+\gamma\in\sigma_u(\mathcal{A})$. If $u$ and $v$ are strongly cospectral with respect to $\textbf{A}$ and $\mathcal{A}$, then $\alpha+(\beta+\gamma) k\in\sigma_{uv}^+(\textbf{A}(X))$ and $\alpha+\gamma\in\sigma_{uv}^+(\mathcal{A})$.
\item If $\beta=-\gamma$, then $\alpha\in\sigma_u(\textbf{A})$, while if $\alpha=-\gamma$, then $0\in\sigma_u(\mathcal{A})$. If $u$ and $v$ are strongly cospectral with respect to $\textbf{A}$ and $\mathcal{A}$, then $\alpha\in\sigma_{uv}^+(\textbf{A})$ and $0\in\sigma_{uv}^+(\mathcal{A})$.
\end{enumerate}
\end{enumerate} 
\end{proposition}
\begin{proof}
To prove (1), suppose $\sigma_u(H)$ has only one element, say $\lambda_1$. Since the $E_j$'s sum to identity, $\textbf{e}_u=\sum_{j}E_{j}\textbf{e}_u
=E_{1}\textbf{e}_u$, and hence, $H\textbf{e}_u=\sum_{j}\lambda_j E_{j}\textbf{e}_u=\lambda_1 E_{1}\textbf{e}_u=\lambda_1\textbf{e}_u$. Equivalently, $(H)_{u,j}=0$ for $j\neq u$, which contradicts the irreducibility of $H$. Hence, $|\sigma_u(H)|\geq 2$. Next, since $X$ is connected, $M$ is an irreducible Hermitian matrix, and so $|\sigma_u(M)|\geq 2$. Applying the Perron Frobenius Theorem yields \ref{posneg}. Let us prove 2. If $X$ is weighted $k$-regular, then $\alpha+(\beta+\gamma) k$ and $\alpha+\gamma$ are simple eigenvalues of $\textbf{A}$ and $\mathcal{A}$, resp., both with eigenvector $\textbf{1}$. On the other hand, if $\beta=-\gamma$ (resp., $\alpha=-\gamma$), then the connectedness of $X$ then implies that $\alpha$ (resp., $0$) is a simple eigenvalue of $\textbf{A}$ (resp., $\mathcal{A}$) with eigenvector $\textbf{1}$ (resp., $D^{\frac{1}{2}}\textbf{1}$).
\end{proof}

We now look at the implications of having twins in $X$ to the eigenvalues of $M$.

\begin{lemma}
\label{eu-ev}
The vector $\textbf{e}_u-\textbf{e}_v$ is an eigenvector of $M$ associated to $\theta$ if and only if
\begin{enumerate}
\item \label{eu-ev1} $\theta=\alpha+\beta\operatorname{deg}(u)+\gamma[(A)_{u,u}-(A)_{u,v}]$, $(2\beta+\gamma)[(A)_{u,u}-(A)_{v,v}]=0$ and $\left(A\right)_{j,u}=\left(A\right)_{j,v}$ for each $j\neq u,v$, whenever $M(X)=\textbf{A}(X)$; and
\item \label{eu-ev2} $\theta=\alpha+\gamma\left(\frac{(A)_{u,u}}{\operatorname{deg}(u)}-\frac{(A)_{u,v}}{\sqrt{\operatorname{deg}(u)\operatorname{deg}(v)}}\right)$, $\frac{(A)_{u,u}}{\operatorname{deg}(u)}=\frac{(A)_{v,v}}{\operatorname{deg}(v)}$ and $\frac{(A)_{j,u}}{\sqrt{\operatorname{deg}(u)}}=\frac{(A)_{j,v}}{\sqrt{\operatorname{deg}(v)}}$ for each $j\neq u,v$, whenever $M=\mathcal{A}$.
\end{enumerate}
In particular, if $T(\omega,\eta)$ is a set of twins in $X$ and $u,v\in T(\omega,\eta)$, then $\textbf{e}_u-\textbf{e}_v$ is an eigenvector for $M$ associated to $\theta$, where
\begin{equation}
\label{eval}
\theta=
\begin{cases}
 \alpha+\beta\operatorname{deg}(u)+\gamma(\omega-\eta), &\text{if $M=\textbf{A}$}\\
 \alpha+\frac{\gamma(\omega-\eta)}{\operatorname{deg}(u)}, &\text{if $M=\mathcal{A}$}.
\end{cases}
\end{equation}
\end{lemma}

\begin{proof}
Let $\textbf{e}_u-\textbf{e}_v$ be an eigenvector for $M$ associated to an eigenvalue $\theta$. If $M=\textbf{A}$, then we obtain
\begin{equation}
\label{adjtwin}
\textbf{A}(\textbf{e}_u-\textbf{e}_v)=\theta(\textbf{e}_u-\textbf{e}_v).
\end{equation}
Comparing $j$th entries of (\ref{adjtwin}) yields the conclusion for statement \ref{eu-ev1}. Next, consider the case $M=\mathcal{A}$. Then
\begin{equation}
\label{adjtwin1}
D^{-\frac{1}{2}}AD^{-\frac{1}{2}}(\textbf{e}_u-\textbf{e}_v)=\left(\frac{\theta-\beta}{\gamma}\right)(\textbf{e}_u-\textbf{e}_v),
\end{equation}
and again comparing $j$th entries of (\ref{adjtwin1}) proves statement \ref{eu-ev2}. In particular, if $T=T(\omega,\eta)$ is a set of twins in $X$, then $\eta=(A)_{u,v}$, $\omega=(A)_{u,u}=(A)_{v,v}$, and $\operatorname{deg}(u)=\operatorname{deg}(v)$ which yields (\ref{eval}).
\end{proof}

If Lemma \ref{eu-ev}(\ref{eu-ev1})holds, then $(A)_{u,u}=(A)_{v,v}$ is equivalent to $\operatorname{deg}(u)=\operatorname{deg}(v)$, while if Lemma \ref{eu-ev}(\ref{eu-ev2}) holds, then $(A)_{u,u}=(A)_{v,v}\neq 0$ is equivalent to $\operatorname{deg}(u)=\operatorname{deg}(v)$.  We also remark that by Lemma \ref{eu-ev}(\ref{eu-ev1}), $u$ and $v$ are twins in $X$ if and only if $\textbf{e}_u-\textbf{e}_v$ is an eigenvector for $M=A,\ L$ or $Q$. In fact, the converse of the last statement of Lemma \ref{eu-ev} is true for the case $M=\textbf{A}$ whenever $\beta\neq -\frac{\gamma}{2}$. However, it does not necessarily hold for $M=\textbf{A}$ whenever $\beta=-\frac{\gamma}{2}$, as well as for $M=\mathcal{A}$ whenever $\operatorname{deg}(u)\neq \operatorname{deg}(v)$. To illustrate this, let $X\cong C_3(1,0,1,1)$ in Figure \ref{cycle}, and consider $\textbf{A}=\alpha I-\frac{\gamma}{2}D+\gamma A$. Then $\alpha-\frac{\gamma}{2}$ is an eigenvalue of $\textbf{A}$ with associated eigenvector $\textbf{e}_u-\textbf{e}_v$, but $u$ and $v$ are not twins in $X$. Further, if $Y\cong C_3(0,-2,3,6)$, then $\alpha+\gamma$ is an eigenvalue of $\mathcal{A}$ with eigenvector $\textbf{e}_u-\textbf{e}_v$, but $u$ and $v$ are not twins in $Y$. In addition, if $X$ is simple, then Lemma \ref{eu-ev}(\ref{eu-ev1}) implies that $u$ and $v$ are twins in $X$ if and only if $\textbf{e}_u-\textbf{e}_v$ is an eigenvector for $M=\textbf{A}$. However, this does not necessarily hold for $M=\mathcal{A}$. We also observe that $u$ and $v$ are twins in $X$, then we can write $\theta$ in (\ref{eval}) as
\begin{equation*}
\theta=
\begin{cases}
 \omega-\eta, &\text{if $M=A$}\\
\operatorname{deg}(u)-\omega+\eta, &\text{if $M=L$}\\
\operatorname{deg}(u)+\omega-\eta, &\text{if $M=Q$}\\
 1-\frac{\omega-\eta}{\operatorname{deg}(u)}, &\text{if $M=\mathcal{L}$}.
\end{cases}
\end{equation*}

Now, suppose $u$ and $v$ are twins in $X$. Define the function $g$ on $V(X)$ given by $g(u)=v$, $g(v)=u$, and $g(a)=a$ for all $a\in V(X)\backslash \{u,v\}$. Then $g$ is an involution that switches $u$ and $v$ and fixes all other vertices. Conversely, suppose there exists an involution $g$ of $X$ that switches $u$ and $v$ and fixes all other vertices. Let $N(u)\backslash \{u,v\}=\{u_1,\ldots,u_j\}$. Since $f$ fixes all vertices other than $u$ and $v$, we get
\begin{equation*}
N(v)\backslash\{u,v\}=N(g(u))\backslash g(v)=\{g(u_1),\ldots,g(u_j)\}=\{u_1,\ldots,u_j\}=N(u)\backslash \{u,v\}.
\end{equation*}
As $g$ preserves the weights of adjacent vertices, $u$ and $v$ are twins. This yields the following facts.

\begin{lemma}
\label{autw}
Vertices $u$ and $v$ are twins in $X$ if and only if there exists an involution on $X$ that switches $u$ and $v$ and fixes all other vertices.
\end{lemma}

\begin{corollary}
\label{cosptw}
If $u$ and $v$ are twins in $X$, then $u$ and $v$ are cospectral with respect to $M$.
\end{corollary}

%%%%%%%%%%%%%%%%%%%%%%%%%%%%%%%%%%%%%%%

\section{Strong cospectrality}\label{strongcospectrality}

First, we provide a characterization of strong cospectrality between columns of Hermitian matrices. We follow the work of Godsil and Smith \cite{Godsil2017}.

\begin{theorem}
\label{sciffcp}
Let $H$ be an $m\times m$ Hermitian matrix. The following are equivalent.
\begin{enumerate}
\item\label{1} $u$ and $v$ are strongly cospectral with respect to $H$.
\item\label{2} $u$ and $v$ are cospectral and parallel with respect to $H$.
\item \label{4} $\phi_u(H,t)=\phi_v(H,t)$ and the poles of $\frac{\phi_S(H,t)}{\phi(H,t)}$ are simple, where $S=\{u,v\}$.
\end{enumerate}
\end{theorem}

The equivalence of \ref{1} and \ref{4} is an easy exercise, while \ref{2} and \ref{4} are equivalent by Theorems \ref{cosp} and \ref{parchar}. Next, we give another characterization of strong cospectrality, which generalizes \cite[Theorem 11.2]{Godsil2017}.

\begin{theorem}
\label{sciffcp1}
Let $H$ be an $m\times m$ Hermitian matrix. The following are equivalent.
\begin{enumerate}
\item\label{1} $u$ and $v$ are strongly cospectral with respect to $H$.
\item\label{3} There is a unitary matrix $R$ such that $R\textbf{e}_u=\textbf{e}_v$, $RH=HR$ and $R=f(H)$ for some function $f$ that is analytic on its domain $D$ and each eigenvalue of $H$ is contained in $D$.
\end{enumerate}
\end{theorem}

\begin{proof}
Let $u$ and $v$ be strongly cospectral columns of $H$, and $H=\sum_j\lambda_jE_j$ be a spectral decomposition of $H$. Let $f$ be a function that is analytic on its domain $D$ such that each $\lambda_j$ is contained in $D$ and $f(\lambda_j)=c_j$ for each $j$, where $c_j$ is a unit complex number such that $E_j\textbf{e}_v=c_jE_j\textbf{e}_u$. Then $R=f(H)$ is a matrix with the desired properties. The converse is straightforward.
\end{proof}

There are many analytic functions $f$ that satisfy Theorem \ref{sciffcp1}(\ref{3}). In fact, since each $E_j$ is a polynomial in $H$ and $R$ is a sum of $c_j E_j$'s, we may take $f\in\mathbb{C}[x]$. If $H$ is real symmetric, then we may take $f\in\mathbb{R}[x]$.

Using Theorem \ref{sciffcp}(\ref{4}) and following the proof of \cite[Lemmas 2.4 and 2.5]{CG}, we obtain the computational complexity of deciding whether two columns of a Hermitian matrix exhibit strong cospectrality.

\begin{theorem}
\label{polyt}
Let $H$ be an $m\times m$ Hermitian matrix with columns $u$ and $v$. The eigenvalues in $\sigma_u(H)$ can be calculated in polynomial time. Moreover, deciding whether $u$ and $v$ are cospectral, parallel, and strongly cospectral with respect to $H$ can be done in polynomial time.
\end{theorem}

Next, we give a lower bound on the sizes of eigenvalue supports of strongly cospectral vertices.

\begin{theorem}
\label{esupp+}
Let $m\geq 3$, $H$ be an $m\times m$ irreducible Hermitian matrix. If $u$ and $v$ are strongly cospectral with respect to $H$, then $|\sigma_u(H)|\geq 3$. If $H$ is real symmetric, then the following also hold.
\begin{enumerate}
\item  \label{yeh} $|\sigma_{uv}^+(H)|=1$ if and only if $\textbf{e}_u+\textbf{e}_v$ is an eigenvector for $H$. In particular, if $X$ is a connected weighted graph with possible loops, then $|\sigma_{uv}^+(M)|=1$ if and only if $\textbf{e}_u+\textbf{e}_v$ is an eigenvector for $M$ associated to $\lambda$ defined in (\ref{eval2}). Moreover, if $|\sigma_{uv}^+(M)|=1$ then $X$ has positive and negative edge weights, and $u$ and $v$ are not twins.
\item \label{ye} $|\sigma_{uv}^-(H)|=1$ if and only if $\textbf{e}_u-\textbf{e}_v$ is an eigenvector for $H$. In particular, if $X$ is a connected weighted graph with possible loops, then $|\sigma_{uv}^-(M)|=1$ if and only if $\textbf{e}_u-\textbf{e}_v$ is an eigenvector for $H$ associated to $\theta$ defined in (\ref{eval}). Moreover, if $|\sigma_{uv}^-(\textbf{A})|=1$, then $u$ and $v$ are twins if and only if either the loops on $u$ and $v$ have equal weights or $\beta\neq-\frac{\gamma}{2}$, while if $|\sigma_{uv}^-(\mathcal{A})|=1$, then $u$ and $v$ are twins if and only if the loops on $u$ and $v$ have equal nonzero weights.
\end{enumerate} 
\end{theorem}
\begin{proof}
Let $H$ be an $m\times m$ Hermitian matrix, and $u$ and $v$ be strongly cospectral columns of $H$. Then $E_1\textbf{e}_u=c_1E_1\textbf{e}_v$ and $E_2\textbf{e}_u=c_2E_2\textbf{e}_v$ for some unit complex numbers $c_1$ and $c_2$. By way of contradiction, assume $|\sigma_u(H)|=2$ with $\lambda_1,\lambda_2\in \sigma_u(H)$. Without loss of generality, suppose $u$ and $v$ are the first and second columns of $H$, respectively. Since $E_1$ and $E_2$ are Hermitian, we may write 
\begin{equation*}
E_1=\left[\begin{array}{c|c}
\ x_1\ \ \ \ c_1x_1 & \ \overline{x_3}\ \ \cdots\ \ \overline{x_m} \\
\overline{c_1x_1}\ \ |c_1|^2\overline{x_1} & \ \overline{c_1x_3}\ \cdots\ \overline{c_1x_m}\\
\hline
\ x_3\ \ \ \ c_1x_3 &\\
\vdots\ \ \ \ \ \ \ \vdots &\ *\\
\ \ x_m\ \ \ c_1x_m &\\
\end{array}
\right]
\quad \text{and}\quad E_2=\left[\begin{array}{c|c}
\ y_1\ \ \ \ c_2y_1 & \ \overline{y_3}\ \ \cdots\ \ \overline{y_m} \\
\overline{c_2y_1}\ \ |c_2|^2\overline{y_1} & \ \overline{c_2y_3}\ \cdots\ \overline{c_2y_m}\\
\hline
\ y_3\ \ \ \ c_2y_3 &\\
\vdots\ \ \ \ \ \ \ \vdots &\ *\\
\ \ y_m\ \ \ c_2y_m &\\
\end{array}
\right].
\end{equation*}
As $E_j\textbf{e}_u=0$ for all $j\geq 3$, we have $E_j=\left[\begin{array}{c|c} 0 & 0\\ \hline 0 &*\\ \end{array} \right]$. Since the $E_j$'s sum to identity, we get $x_1+y_1=1$, $c_1x_k+c_2y_k=0$ for all $k\neq 2$, and $x_k+y_k=0$ for all $k\geq 3$. The latter two equations imply that $(c_1-c_2)x_k=0$ for all $k\geq 3$. If $c_1=c_2$, then $x_1+y_1=0$, a contradiction. Thus, $c_1\neq c_2$, and so $x_k=y_k=0$ for all $k\geq 3$. Hence, each $E_j$ is block diagonal, and so $H$ is also block diagonal, a contradiction to the irreducibility of $H$. Thus, $|\sigma_u(H)|\geq 3$. Now, suppose further that $H$ is real symmetric. Define
\begin{equation}
\label{sctw2}
\textbf{w}^+:=\displaystyle\sum_{\lambda_k\in\sigma_{uv}^+(M)}E_{k}\textbf{e}_u\quad \text{and}\quad \textbf{w}^-:=\displaystyle\sum_{\lambda_{\ell}\in\sigma_{uv}^-(M)}E_{\ell}\textbf{e}_u.
\end{equation}
Using (\ref{sctw2}), and the fact that the $E_k$'s and $E_{\ell}$'s sum to identity, we get $\textbf{e}_u=\textbf{w}^++\textbf{w}^-$ and $\textbf{e}_v=\textbf{w}^+-\textbf{w}^-$, and so $\textbf{w}^+=\frac{1}{2}(\textbf{e}_u+\textbf{e}_v)\quad \text{and}\quad \textbf{w}^-=\frac{1}{2}(\textbf{e}_u-\textbf{e}_v)$. Thus, both $\sigma_{uv}^+(H)$ and $\sigma_{uv}^-(H)$ have at least one element.
%Now, let us look at the size of the sets $\sigma_{uv}^+(M)$, $\sigma_{uv}^-(M)$ and $\sigma_u(M)$.

Assume $|\sigma_{uv}^+(H)|=1$. From (\ref{sctw2}), $\textbf{w}^+=E_{k}\textbf{e}_u$ for some $k$, and thus,
\begin{equation}
\label{yup}
H\textbf{w}^+=\sum_{j}\lambda_j E_{j}(E_{k}\textbf{e}_u)=\lambda_k E_{k}^2\textbf{e}_u=\lambda_k E_{k}\textbf{e}_u=\lambda_k \textbf{w}^+.
\end{equation}
Using (\ref{yup}), we have
\begin{equation}
\label{m}
H(\textbf{e}_u+\textbf{e}_v)=\lambda_k(\textbf{e}_u+\textbf{e}_v).
\end{equation}
Equivalently, $\lambda:=\lambda_k$ is an eigenvalue for $H$ with eigenvector $\textbf{e}_u+\textbf{e}_v$. On the other hand, if $|\sigma_{uv}^-(H)|=1$, then an argument similar to (\ref{yup}) yields
\begin{equation}
\label{boink}
H(\textbf{e}_u-\textbf{e}_v)=\lambda_{\ell}(\textbf{e}_u-\textbf{e}_v).
\end{equation}
Equivalently $\textbf{w}^-$ is an eigenvector for $H$ associated to $\theta:=\lambda_{\ell}$. Now, if $|\sigma_{uv}^+(H)|=|\sigma_{uv}^-(H)|=1$, then (\ref{m}) and (\ref{boink}) implies that $(H)_{j,u}=0$ for all $j\neq u,v$, which contradicts the irreducibility of $H$. Thus, $|\sigma_{uv}^+(H)|$ and $|\sigma_{uv}^-(H)|$ cannot be both one, and so $|\sigma_u(H)|=|\sigma_{uv}^+(H)|+|\sigma_{uv}^-(H)|\geq 3$.

Now, $|\sigma_{uv}^+(M)|=1$ if and only if $\textbf{e}_u+\textbf{e}_v$ is an eigenvector for $M$ associated to $\lambda$, and $|\sigma_{uv}^-(M)|=1$ if and only if $\textbf{e}_u-\textbf{e}_v$ is an eigenvector for $M$ associated to $\theta$. If $H=M$, comparing entries in (\ref{m}) yields
\begin{equation}
\label{eval2}
\lambda=
\begin{cases}
 \alpha+\beta\operatorname{deg}(u)+\gamma((A)_{u,u}+(A)_{u,v}), &\text{if $M=\textbf{A}$}\\
 \beta+\gamma\left(\frac{(A)_{u,u}}{\operatorname{deg}(u)}+\frac{(A)_{u,v}}{\sqrt{\operatorname{deg}(u)\operatorname{deg}(v)}}\right), &\text{if $M=\mathcal{A}$},
\end{cases}
\end{equation}
$(M)_{j,u}=-(M)_{j,v}$ for $j\neq u,v$, and $(M)_{u,u}=(M)_{v,v}$, which proves (\ref{yeh}). Moreover, since (\ref{boink}) holds with $H=M$, an application of Lemma \ref{eu-ev} yields (\ref{ye}).
\end{proof}

By Theorem \ref{esupp+}, if $|\sigma_u(H)|=2$, then $u$ cannot be strongly cospectral with any column in $H$. We also remark that some results in Theorem \ref{esupp+} relating to $M=L$ in the context of simple unweighted graphs were observed by Coutinho et al. \cite[Lemma 3.1]{Coutinho2014a} in 2014, and then by Chan et al. \cite[Corollary 6.3]{Chan2020} in 2020.

To illustrate Theorem \ref{esupp+}, we give the following example.

\begin{example}
\label{negex}
Consider the graph $Y\cong Y(1,-1)$ in Figure \ref{plantsie} where the vertices marked blue are labelled $u$ and $v$, while the other two are labelled $a$ and $b$. If we index the first two columns of $A$ by $u$ and $v$, then the eigenvalues of $A$ are $1\pm\sqrt{5}$ and $0$ (multiplicity two), with eigenvectors $\left(\frac{1}{2}\left(1\pm\sqrt{5}\right),\frac{1}{2}\left(-1\mp\sqrt{5}\right),1,1\right)^T$ and $\textbf{e}_u+\textbf{e}_v$, $\textbf{e}_a-\textbf{e}_b$. Thus, $E_{0}\textbf{e}_u=E_{0}\textbf{e}_v$ and $E_{1\pm\sqrt{5}}\textbf{e}_u=-E_{{1\pm\sqrt{5}}}\textbf{e}_v$
while $E_{0}\textbf{e}_a=-E_{0}\textbf{e}_b$ and $E_{1\pm\sqrt{5}}\textbf{e}_u=E_{{1\pm\sqrt{5}}}\textbf{e}_v$. Hence, $u$ and $v$ are adjacency strongly cospectral, as are $a$ and $b$. Further, $|\sigma_{uv}^+(A)|=|\sigma_a^-(A)|=1$, $|\sigma_{uv}^-(A)|=|\sigma_a^+(A)|=2$ so that $|\sigma_u(A)|=|\sigma_a(A)|=3$. By Theorem \ref{esupp+}, $X$ has positive and negative edge weights, $\textbf{e}_u+\textbf{e}_v$ is an eigenvector for $A$, and $u$ and $v$ are not twins, while $a$ and $b$ are.
\end{example}

The following theorem states that cospectrality, parallelism, and strong cospectrality in simple unweighted graphs are closed under complementation under some conditions.

\begin{theorem}
\label{comp}
Let $X$ be a simple connected unweighted graph. The following statements hold.
\begin{enumerate}
\item \label{compA} If $X$ is regular, then two vertices in $X$ are strongly cospectral with respect to $M(X)$ if and only if they are strongly cospectral with respect to $M(\overline{X})$.
\item If $\beta=-\gamma$, then two vertices in $X$ are strongly cospectral with respect to $\textbf{A}(X)$ if and only if they are strongly cospectral with respect to $\textbf{A}(\overline{X})$.
\end{enumerate}
\end{theorem}

\begin{proof}
If $X$ is regular, then $\textbf{A}(X)=(\alpha+\beta k)I+\gamma A(X)\ \text{and}\ \textbf{A}(\overline{X})=(\alpha+\beta k-\gamma)I+\gamma J-\gamma A(X)$. Similarly, $\mathcal{A}(X)=\alpha I+\frac{\gamma}{k}A(X)\ \text{and}\ \mathcal{A}(\overline{X})=\left(\alpha-\frac{\gamma}{n-k-1}\right) I+\left(\frac{\gamma}{n-k-1}\right)J-\left(\frac{\gamma}{n-k-1}\right)A(X)$. Since $\textbf{1}$ is an eigenvector for $A(X)$, these equations imply that $\textbf{A}(X)$ and $\textbf{A}(\overline{X})$, as well as $\mathcal{A}(X)$ and $\mathcal{A}(\overline{X})$, have the same set of orthogonal projection matrices in their spectral decompositions. Thus, (\ref{compA}) holds. If $\beta=-\gamma$, then the same argument holds as $\textbf{1}$ is an eigenvector for both $L(X)$ and $L(\overline{X})$.
\end{proof}

Next, we give another property of strong cospectrality when the graph is bipartite. For brevity, we write $\textbf{A}_{-\gamma}=\alpha I+\beta D-\gamma A$ and $\mathcal{A}_{-\gamma}=\alpha I-\gamma D^{-\frac{1}{2}}D^{-\frac{1}{2}}$, and use $M_{-\gamma}$ to denote either $\textbf{A}_{-\gamma}$ or $\mathcal{A}_{-\gamma}$.

\begin{theorem}
\label{bip}
Let $X$ be a simple connected weighted bipartite graph. Then $u$ and $v$ are strongly cospectral with respect to $M$ if and only if they are strongly cospectral with respect to $M_{-\gamma}$. In particular, $\sigma_{uv}^+(M)=\sigma_{uv}^+(M_{-\gamma})$ if $u$ and $v$ belong to the same partite set in $X$, and $\sigma_{uv}^+(M)=\sigma_{uv}^-(M_{-\gamma})$ otherwise.
\end{theorem}

\begin{proof}
Let $X_1$ and $X_2$ be the partite sets of $X$. Define the diagonal matrix $S$ with $(S)_{u,u}=1$ if $u\in X_1$ and $(S)_{u,u}=-1$ otherwise. Then $S^2=I$ and $SAS=-A$, and so $S\textbf{A}S=\textbf{A}_{-\gamma}\quad \text{and}\quad S\mathcal{A}S=\mathcal{A}_{-\gamma}$. Let $\textbf{A}=\sum_j\lambda_jE_j$ be the spectral decomposition of $\textbf{A}$. Then $\textbf{A}_{-\gamma}=\sum_j\lambda_jSE_jS$ is the spectral decomposition of $\textbf{A}_{-\gamma}$. If $u$ and $v$ are strongly cospectral with respect to $\textbf{A}$, then $E_j\textbf{e}_u=\pm E_j\textbf{e}_v$ for each $j$. Since $S\textbf{e}_u=\pm \textbf{e}_u$, we get that $SE_jS\textbf{e}_u=\pm SE_jS\textbf{e}_v$. In particular, if $E_j\textbf{e}_u=E_j\textbf{e}_v$, then  $SE_jS\textbf{e}_u=E_jS\textbf{e}_v$ whenever $u$ and $v$ are in the same partite set in $X$. Otherwise, $SE_jS\textbf{e}_u=-E_jS\textbf{e}_v$. The same argument applies to $\mathcal{A}$.
\end{proof}

An immediate consequence of Theorem \ref{bip} is that strongly cospectrality with respect to the Laplacian and signless Laplacian matrices are equivalent for bipartite graphs.

Now, combining Lemma \ref{sciffcp} and Corollary \ref{cosptw} yields the next result.

\begin{corollary}
\label{par}
Let $u$ and $v$ be twins in $X$. With respect to $M$, $u$ and $v$ are strongly cospectral if and only if they are parallel.
\end{corollary}

For vertices with equal eigenvalue supports containing simple eigenvalues, strong cospectrality and cospectrality are equivalent, while for twin vertices, strong cospectrality and parallelism are equivalent.

Let us now characterize twin vertices that are strongly cospectral.
 
\begin{theorem}
\label{partw}
Let $T=T(\omega,\eta)$ be a set of twins in $X$. Consider $\theta$ defined in (\ref{eval}). The following statements hold with respect to $M$.
\begin{enumerate}
\item \label{contri} Let $u,v\in T$. If
$\Omega$ is an orthogonal set of eigenvectors for $\theta$ such that $\textbf{e}_u-\textbf{e}_v\in \Omega$, then $E_{\theta}\textbf{e}_u=cE_{\theta}\textbf{e}_v$ if and only if $c=-1$. Moreover, if $E_{\theta}\textbf{e}_u=-E_{\theta}\textbf{e}_v$, then $|T|=2$ and either $|\Omega|=1$ or $\textbf{w}^T\textbf{e}_u=\textbf{w}^T\textbf{e}_v=0$ for all $\textbf{w}\in\Omega\backslash\{\textbf{e}_u-\textbf{e}_v\}$.
\item \label{notpar} If $u\in T$ and $v\in V(X)\backslash T$, then $E_{\theta}\textbf{e}_u\neq cE_{\theta}\textbf{e}_v$ for any $c\in\mathbb{R}$.
\item \label{notparr} For all $\mu\in\sigma_u(M(X))$ with $\mu\neq \theta$, $E_{\mu}\textbf{e}_u=E_{\mu}\textbf{e}_v$ for all $u,v\in T$.
\end{enumerate}
\end{theorem}

\begin{proof}
Let $m=|T|$, and without loss of generality, suppose the first $|T|$ columns of $M=M(X)$ are indexed by the elements of $T$. By Lemma \ref{eu-ev}, $\{\textbf{e}_1-\textbf{e}_j:j=2,\ldots,m\}$ is a set of eigenvectors of $M$ corresponding to the eigenvalue $\theta$ defined in (\ref{eval}). Orthogonalizing this set yields an orthogonal subset $W=\left\{\textbf{e}_1+\ldots+\textbf{e}_{j-1}-(j-1)\textbf{e}_j:j=2,\ldots,m\right\}$ of eigenvectors for $M$ corresponding to $\theta$. Let $\Omega'$ be an orthogonal set of eigenvectors for $M$, and $\Omega$ be an orthogonal set of eigenvectors for $M$ corresponding to $\theta$ such that $W\subseteq \Omega$. For each $\textbf{w}=(x_1,\ldots,x_n)\in\Omega'\backslash W$, $\textbf{w}\cdot (\textbf{e}_1+\ldots+\textbf{e}_{j-1}-(j-1)\textbf{e}_j)=0$ for every $j=2,\ldots,m$, and so we get
\begin{equation}
\label{bruno}
\textbf{w}=(x,\ldots,x,x_{m+1},\ldots,x_n),
\end{equation}
for some $x,x_{m+1},\ldots, x_n\in\mathbb{R}$. Suppose $W\neq \Omega$ and let $\textbf{w}\in \Omega\backslash W$. If $x\neq 0$ and $x_j=0$ for all $j=m+1,\ldots,n$, then $[\textbf{1}_{m}\ \textbf{ 0}_{n-m}]^T$ is an eigenvector for $M$ corresponding to $\theta$. Observe that we can write $A=\left[
\begin{array}{cc} A_1 & A_2 \\ A_3 & A_4 \end{array} \right]$ where $A_2=A_3^T=[\delta_{1,m+1}\textbf{1}_{m}\ \cdots\ \delta_{1,n}\textbf{1}_{m}]$, where $\delta_{1,j}=0$ if $j\notin N_X(u)$ for each $u\in T$ and $\delta_{1,j}\neq 0$ otherwise. Thus, $\textbf{A}=\left[
\begin{array}{cc} * & * \\ \gamma A_3 & * \end{array} \right]$, and so $\textbf{A}\left[
\begin{array}{cc} \textbf{1}_{m} \\ \textbf{0}_{n-m} \end{array} \right]=\left[
\begin{array}{cc} * \\ \gamma A_3\textbf{1}_{m} \end{array} \right]=\left[
\begin{array}{cc} * \\ \gamma m(\delta_{1,m+1},\ldots,\delta_{1,n})^T \end{array} \right]=\theta\left[
\begin{array}{cc} \textbf{1}_{m} \\ \textbf{0}_{n-m} \end{array} \right]$. Since $m,\gamma\neq 0$, it follows that $\delta_{1,j}=0$ for each $j=m+1,\ldots,n$ so that $X$ is disconnected, which is a contradiction. Applying the same argument to $\mathcal{A}$ also yields the same result. Thus, $x=0$ or $x_j\neq 0$ for at least one $j$ so that if $\textbf{w}\in \Omega\backslash W$, then either $\textbf{w}=(0,\ldots,0,x_{m+1},\ldots,x_n)$ or $\textbf{w}=(1,\ldots,1,x_{m+1},\ldots,x_n)$. This allows us to write $E_{\theta}=E_{W}+E_{\Omega\backslash W}$, where
\begin{equation}
\label{uGH}
E_{W}=\left(I_{m}-\frac{1}{m}\textbf{J}_{m}\right)\oplus \textbf{0}_{n-m}
\end{equation}
and
\begin{equation}
\label{uGH1}
E_{\Omega\backslash W}=\sum_{\textbf{z}\in Z}\frac{1}{\|\textbf{z}\|^2}\left[
\begin{array}{cc} \textbf{0} & \textbf{0} \\ \textbf{0} & \textbf{z}\textbf{z}^T \end{array} \right]+\sum_{\textbf{z}\in Z'}\frac{1}{\|\textbf{z}\|^2}\left[
\begin{array}{cc}J_m & \textbf{1}_{m}\textbf{z}^T \\ (\textbf{z})(\textbf{1}_{m}^T) & \textbf{z}\textbf{z}^T \end{array} \right],
\end{equation}
where $Z=\{\textbf{z}:(\textbf{0}_{m},\textbf{z})\in \Omega\backslash W\}$, $Z'=\{\textbf{z}:
(\textbf{1}_{m},\textbf{z})\in \Omega\backslash W\}$ and $E_{\Omega\backslash W}$ is absent if $W=\Omega$. Note that $Z$ or $Z'$ can be empty, and in case they are nonempty, then they are linearly independent sets.

\begin{figure}[h!]
	\begin{center}
		\begin{tikzpicture}
		\tikzset{enclosed/.style={draw, circle, inner sep=0pt, minimum size=.22cm}}
	   
		\node[enclosed, fill=cyan, label={left, yshift=0cm: $u$}] (v_1) at (0,0.5) {};
		\node[enclosed, fill=cyan, label={left, yshift=0cm: $v$}] (v_2) at (0,-0.5) {};
		\node[enclosed] (v_3) at (0.8,0) {};
		\node[enclosed] (v_4) at (1.8,0) {};
		\node[enclosed, fill=cyan, label={right, yshift=0cm: $w$}] (v_5) at (2.6,0.5) {};
		\node[enclosed, fill=cyan, label={right, yshift=0cm: $x$}] (v_6) at (2.6,-0.5) {};
		
		\draw (v_1) --  (v_3);
		\draw (v_2) --  (v_3);
		\draw (v_3) --  (v_4);
		\draw (v_4) --  (v_5);
		\draw (v_4) --  (v_6);
		\end{tikzpicture}
	\end{center}
	\caption{A tree with four pairwise cospectral vertices marked blue}\label{alpha}
\end{figure}
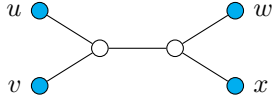

We now prove (1). Let $u\in T$. Using (\ref{uGH}) and (\ref{uGH1}), we obtain
\begin{equation}
\label{UGH2}
E_{\theta}\textbf{e}_u=\left[
\begin{array}{cc} \textbf{y} \\ \textbf{0}_{n-m} \end{array} \right]+\sum_{\textbf{z}\in Z'}\frac{1}{\|\textbf{z}\|^2}\left[\begin{array}{cc} \textbf{1}_m \\ \textbf{z} \end{array} \right]
\end{equation}
where $\textbf{y}=\left(-\frac{1}{m},\ldots,-\frac{1}{m},1-\frac{1}{m},-\frac{1}{m}, \ldots,-\frac{1}{m}\right)^T$ and the entry of $\textbf{y}$ equal to $1-\frac{1}{m}$ is indexed by $u$. Now, let $v\in T\backslash\{u\}$. Using (\ref{UGH2}), we have $E_{\theta}\textbf{e}_u=cE_{\theta}\textbf{e}_v$ for some $c\in\mathbb{R}$ if and only if
\begin{equation*}
c+(c-1)\left(-\frac{1}{m}+\displaystyle\sum_{\textbf{z}\in Z'}\frac{1}{\|\textbf{z}\|^2}\right)=-1+(c-1)\left(-\frac{1}{m}+\displaystyle\sum_{\textbf{z}\in Z'}\frac{1}{\|\textbf{z}\|^2}\right)=0.
\end{equation*}
Equivalently, $c=-1$. Now, comparing the $j$th entries of $E_{\theta}\textbf{e}_u$ and $-E_{\theta}\textbf{e}_v$ for $j\in T$ yields $-\frac{1}{m}=-\frac{m-1}{m}$, which is possible if and only if $m=2$, i.e., $|W|=1$. Moreover, comparing the last $n-m$ entries of $E_{\theta}\textbf{e}_u$ and $-E_{\theta}\textbf{e}_v$ gives us $\sum_{\textbf{z}\in Z'}\frac{1}{\|\textbf{z}\|^2}\textbf{z}=0$. Since $Z'$ is a linearly independent set, it must be that $Z'=\varnothing$. If $Z\neq \varnothing$, then $\textbf{w}^T\textbf{e}_u=\textbf{w}^T\textbf{e}_v=0$ while if $Z=\varnothing$, then $|\Omega|=|W|=1$. This proves (\ref{contri}).

Next, we show (\ref{notpar}). Assume $v\in V(X)\backslash T$. Then (\ref{uGH}) and (\ref{uGH1}) yield
\begin{equation}
\label{UGH3}
E_{\theta}\textbf{e}_v=\sum_{\textbf{z}\in Z}\frac{x_v}{\|\textbf{z}\|^2}\left[\begin{array}{cc} \textbf{0}_m \\ \textbf{z} \end{array} \right]+\sum_{\textbf{z}\in Z'}\frac{x_v}{\|\textbf{z}\|^2}\left[\begin{array}{cc} \textbf{1}_m \\ \textbf{z} \end{array} \right],
\end{equation}
where $x_v$ is the $v$th entry of $\textbf{w}$. If $E_{\theta}\textbf{e}_u=cE_{\theta}\textbf{e}_v$ for some $c\in\mathbb{R}$, then one checks using (\ref{UGH2}) and (\ref{UGH3}) that the entries of $E_{\theta}\textbf{e}_u$ and $cE_{\theta}\textbf{e}_v$ indexed by $T$ yield $-\frac{1}{m}=\frac{m-1}{m}$, a contradiction.

Finally, from Proposition \ref{esupp}, $|\sigma_u(M)|\geq 2$, and so there exists $\mu\in \sigma_u(M)$ with $\mu\neq \theta$. Let $\textbf{w}_1,\ldots,\textbf{w}_s\in\Omega'$ be eigenvectors corresponding to $\mu$. From (\ref{bruno}), $\textbf{w}_j^T\textbf{e}_u=\textbf{w}_j^T\textbf{e}_v$ for all $u,v\in T$. Thus, (\ref{notparr}) holds.
\end{proof}

By Theorem \ref{partw}(\ref{notpar}), a vertex with a twin cannot be parallel to a vertex that is not its twin. However, it is possible for a vertex with a twin to be cospectral with respect to $M$ to a vertex that is not its twin. In Figure \ref{alpha}, $u$ and $v$ are false twins, $w$ is not twins with $u$ and $u$, $v$, and $w$ are pairwise cospectral with respect to $M$. In \cite{Godsil2017}, Godsil and Smith posed an interesting question: is there a simple unweighted tree with three pairwise adjacency strongly cospectral vertices? Emanuel Silva has noted in a private communication that up to 22 vertices, the answer to this question is negative. Nevertheless, the problem remains open.

Combining Theorem \ref{partw} statements 1 and 2 yields the following corollary.

\begin{corollary}
\label{3tw}
Let $T=T(\omega,\eta)$ be a set of twins in $X$. If $|T|\geq 3$, then each $u\in T$ is not parallel, and hence not strongly cospectral, with any $v\in V(X)\backslash\{u\}$ with respect to $M$.
\end{corollary}

Combining Corollary \ref{3tw} with Lemma \ref{autw}, we see that a graph with involution does not automatically give rise to pairs of strongly cospectral vertices, contrary to the claim of Kempton, Lippner and Yau \cite{Kempton}.

Strong cospectrality in general is not a monogamous property. Indeed, the four degree two vertices in $K_2\square P_3$ (see Figure \ref{KP}) are pairwise strongly cospectral with respect to $M=\textbf{A}$. However, by Corollary \ref{3tw}, strong cospectrality that involves a vertex with a twin is monogamous. Following the proof of \cite[Lemma 4.3]{Godsil2017}, we get an upper bound for the number of columns in a Hermitian matrix that are strongly cospectral.

\begin{theorem}
Let $H$ be an $m\times m$ Hermitian matrix and $u\in\{1,\ldots,m\}$. Then the number of distinct columns $v$ such that $u$ and $v$ are parallel with respect to $H$ and $\sigma_v(H)=\sigma_u(H)$ is at most $|\sigma_u(H)|-1$.
\end{theorem}

We mention an interesting result of {\'{A}}rnad{\'{o}}ttir and Godsil which states that if $X$ is a normal Cayley graph and $n$ is the largest multiplicity of an eigenvalue of $A$, then the number of pairwise adjacency strongly cospectral vertices is bounded above by $|V (X)|/n$ \cite[Theorem 6.1]{Arn2021}. It would be interesting to check whether this bound holds for other classes of graphs, and in general whenever $M=\textbf{A}$. Cayley graphs are vertex-transitive, and so any pair of vertices in a Cayley graph are cospectral, which makes them promising candidates for graphs that exhibit strong cospectrality.

Next, we state the following interesting observation.

\begin{proposition}
\label{prop}
Let $H$ be an $m\times m$ Hermitian matrix. If all columns of $H$ are strongly cospectral, then $|(E_j)_{u,v}|=\frac{1}{m}$ and $(E_j)_{u,u}=\frac{1}{m}$ for all $u,v$. If we add that $H$ is real symmetric, then $(E_j)_{u,v}=\pm \frac{1}{m}$ for all $u,v$, and $m$ is even.
\end{proposition}

\begin{proof}
Suppose all columns of $H$ are strongly cospectral. Then for any columns $u$ and $v$ of $H$, we have $E_j\textbf{e}_u=cE_j\textbf{e}_v$ for some $c\in\mathbb{C}$ with $|c|=1$. Thus, $\sigma_u(H)=\sigma_v(H)$ and each $E_r$ is a rank one matrix, which implies that all eigenvalues of $H$ are simple. As a result, if $\{\textbf{v}_1,\ldots,\textbf{v}_m\}$ is an orthonormal set of eigenvectors for $H$, then $E_j=\textbf{v}_j\textbf{v}_j^*$ for each $j$. Since $E_{j}\textbf{e}_u=(\textbf{v}_j^*\textbf{e}_u)\textbf{v}_j$ and $E_{j}\textbf{e}_v=(\textbf{v}_j^*\textbf{e}_v)\textbf{v}_j$, it follows that $E_j\textbf{e}_u=(\textbf{v}_j^*\textbf{e}_u/\textbf{v}_j^*\textbf{e}_v)E_j\textbf{e}_v$. Thus, we may take $c=\textbf{v}_j^*\textbf{e}_u/\textbf{v}_j^*\textbf{e}_v$, which implies that $|\textbf{v}_j^*\textbf{e}_u/\textbf{v}_j^*\textbf{e}_v|=1$, or equivalently, $|\textbf{v}_j^*\textbf{e}_u|=|\textbf{v}_j^*\textbf{e}_v|$. Simply put, every entry of $\textbf{v}_j$ has magnitude $\frac{1}{\sqrt{m}}$, and consequently, every entry of $E_j$ has magnitude $\frac{1}{m}$. In particular, each diagonal entry of $E_j$ is equal to $\frac{1}{m}$. Moreover, if $H$ is real symmetric, then every off-diagonal entry of $E_j$ is equal to $\pm\frac{1}{m}$, and because $E_jE_{\ell}=0$ for $j\neq \ell$, it follows that $m$ must be even.
\end{proof}

\begin{example}
\label{godsil}
Consider the real symmetric matrix
\begin{equation}
\label{Hhh}
H=\left[
\begin{array}{cccc} 0&1&3&0 \\ 1&0&0&3 \\ 3&0&0&1 \\ 0&3&1&0 \end{array} \right].
\end{equation}
The eigenvalues of $H$ are $4$, $-4$, $2$ and $-2$ with associated eigenvectors $\frac{1}{2}\textbf{1}$, $\frac{1}{2}[1,-1,-1,1]^T$, $\frac{1}{2}[1,-1,1,-1]^T$ and $\frac{1}{2}[1,1,-1,-1]^T$ which altogether forms an orthonormal set. Since each eigenvalue of $H$ is simple, every pair of columns of $H$ are parallel. Moreover, one checks that $(E_j)_{u,u}=\frac{1}{4}$ for every column $u$. Consequently, all columns of $H$ are strongly cospectral.
\end{example}

In \cite[Lemma 10.1]{Godsil2017}, Godsil showed that $K_2$ is the only simple unweighted graph that exhibits adjacency strong cospectrality between every pair of vertices. However, this is not true for the weighted case. Indeed, if we take $X\cong C_4(1,3,1,3)$ in Figure \ref{cycle}, then $A=H$ in (\ref{Hhh}). Since $\textbf{A}=(\alpha+4\beta)I+\gamma A$, $\textbf{A}$ and $A$ share the same eigenvectors. By Example (\ref{godsil}), all columns of $\textbf{A}$ are strongly cospectral.

Next, combining Theorem \ref{partw} and Corollary \ref{par}, we acquire a spectral characterization of twin vertices that are strongly cospectral.

\begin{corollary}
\label{strcospchar}
Let $T(\omega,\eta)=\{u,v\}$ be a set of twins in $X$, and consider $\theta$ in (\ref{eval}). If $\Omega$ is an orthogonal set of eigenvectors for $\theta$ such that $\textbf{e}_u-\textbf{e}_v\in \Omega$, then $u$ and $v$ are strongly cospectral with respect to $M$ if and only if either $|\Omega|=1$ or $\textbf{w}^T\textbf{e}_u=\textbf{w}^T\textbf{e}_v=0$ for all $\textbf{w}\in\Omega\backslash\{\textbf{e}_u-\textbf{e}_v\}$. Moreover, if $u$ and $v$ are strongly cospectral with respect to $M$, then $\sigma_{uv}^-(M)=\{\theta\}$, $\sigma_{uv}^+(M)=\sigma_u(M)\backslash \{\theta\}$, and $u$ and $v$ cannot be strongly cospectral with any $w\in V(X)\backslash\{u,v\}$.
\end{corollary}

If the eigenvalue $\theta$ in Corollary \ref{strcospchar} is simple, then we get the following result.

\begin{corollary}
\label{strcospcharsimp}
Let $T(\omega,\eta)=\{u,v\}$ be a set of twins in $X$, and consider $\theta$ in (\ref{eval}). If $\theta$ is a simple eigenvalue of $M$, then $u$ and $v$ are strongly cospectral with respect to $M$, and $E_{\theta}=\frac{1}{2}(\textbf{e}_u-\textbf{e}_v)(\textbf{e}_u-\textbf{e}_v)^T$.
\end{corollary}

\begin{example}
\label{simple}
For $\omega\in\mathbb{R}$, let $X\cong P_3(\omega)$ be the unweighted $P_3$ with end vertices $u$ and $v$, and an added loop on the middle vertex with weight $\omega$. Since $u$ and $v$ are false twins in $X$, Lemma \ref{eu-ev} yields $\theta$ as a simple eigenvalue of $M$ with associated eigenvector $\textbf{e}_u-\textbf{e}_v$, where $\theta$ is given in \ref{eval}. By Corollary \ref{strcospcharsimp}, it follows that $u$ and $v$ are strongly cospectral with respect to $M$ for all $\omega\in\mathbb{R}$.
\end{example}

By taking a pair of non-adjacent vertices in the unweighted $C_4$, it is evident that the converse of Corollary \ref{strcospcharsimp} does not hold. Now, combining Theorem \ref{esupp+}(\ref{ye}) and Corollary \ref{strcospchar}, we obtain a characterization of strongly cospectral vertices which are also twins.

\begin{corollary}
\label{strcosptw}
Let $u$ and $v$ be strongly cospectral with respect to $M$. If $u$ and $v$ are twins, then $|\sigma_{uv}^-(M)|=1$. Conversely, if one of the following conditions hold then $u$ and $v$ are twins.
\begin{enumerate}
\item $|\sigma_{uv}^-(\textbf{A})|=1$ and either the loops on $u$ and $v$ have equal weights or $\beta\neq-\frac{\gamma}{2}$.
\item $|\sigma_{uv}^-(\mathcal{A})|=1$ and the loops on $u$ and $v$ have equal weights.
\end{enumerate}
\end{corollary}

The following result is a consequence of Theorem \ref{esupp+} and Corollary \ref{strcospchar}, which gives a lower bound for the sizes of the eigenvalue supports of strongly cospectral vertices given some spectral information.

\begin{corollary}
\label{nottw}
Let $m\geq 3$, $H$ be an $m\times m$ irreducible Hermitian matrix.
\begin{enumerate}
\item If $\textbf{e}_u+\textbf{e}_v$ and $\textbf{e}_u-\textbf{e}_v$ are not eigenvectors for $H$, then $\sigma_u(H)\geq 4$.
\item If either $\textbf{e}_u+\textbf{e}_v$ is an eigenvector for $H$ but $\textbf{e}_u-\textbf{e}_v$ is not, or $\textbf{e}_u-\textbf{e}_v$ is an eigenvector for $H$ but $\textbf{e}_u+\textbf{e}_v$ is not, then $\sigma_u(H)\geq 3$.
\end{enumerate}
\end{corollary}

%%%%%%%%%%%%%%%%%%%%%%%%%%%%%%%%%%%%%%%%%%%

\section{Graph products}\label{prod}

In this section, we determine when strong cospectrality is preserved under Cartesian and direct products of graphs. For two connected weighted graphs $X$ and $Y$ with possible loops, we denote their \textit{Cartesian product} by $X\square Y$, which is the graph with vertex set $V(X)\times V(Y)$ where $(u,x)$ and $(v,y)$ are adjacent in $X\square Y$ if either $x=y$ and $(u,v)$ is an edge in $X$ or $u=v$ and $(x,y)$ is an edge in $Y$. The weight of the edge between $(u,x)$ and $(v,y)$ is equal to the weight of $(u,v)$ if $x=y$ and $(x,y)$ if $u=v$. Moreover, if $u\in V(X)$ and $x\in V(Y)$ have loops of weight $\omega$ and $\omega'$ in $X$ and $Y$ respectively, then $(u,x)$ also has a loop of weight $\omega+\omega'$ in $X\square Y$. We also denote the \textit{direct product} of $X$ and $Y$ by $X\times Y$, which is the graph with vertex set $V(X)\times V(Y)$ where $(u,x)$ and $(v,y)$ are adjacent in $X\times Y$ if $(u,v)$ and $(x,y)$ are edges in $X$ and $Y$, respectively. The weight of the edge between $(u,x)$ and $(v,y)$ is equal to the product of the weights of the edges $(u,v)$ and $(x,y)$. If $x\neq y$ and $(x,y)$ is an edge in $Y$, then $(u,x)$ and $(u,y)$ are adjacent in $X\times Y$ if and only if $u$ has a loop in $X$, and $(u,x)$ has a loop in $X\times Y$ if and only if $u$ and $x$ have loops in $X$ and $Y$, respectively. For more about Cartesian and direct products of simple unweighted graphs, see \cite{ProdGraphs}. It is known that
\begin{equation}
\label{Cart}
D(X\square Y)=\left(D(X)\otimes I\right)+\left(I\otimes D(Y)\right)\quad \text{and}\quad A(X\square Y)=\left(A(X)\otimes I\right)+\left(I\otimes A(Y)\right)
\end{equation}
while one checks that
\begin{equation}
\label{weak}
D(X\times Y)=\left(D(X)\otimes D(Y)\right)-\left(D_1\otimes R_1\right)-\left(R_2\otimes D_2\right)\quad \text{and}\quad A(X\times Y)=A(X)\otimes A(Y),
\end{equation}
where $D_1$ and $D_2$ are the diagonal matrices consisting of the diagonal entries of $A(X)$ and $A(Y)$, while $R_1$ and $R_2$ are the diagonal matrices consisting of the row sums of $A(X)$ and $A(Y)$, respectively. If we add that $X$ and $Y$ are simple, then $D(X\times Y)=D(X)\otimes D(Y)$.

We now determine when strong cospectrality is preserved under Cartesian and direct products. We abuse notation and denote the orthogonal projection matrix corresponding to the eigenvalue $\lambda$ by $E_{\lambda}$.

\begin{theorem}
\label{closed}
Let $u$ and $v$ be strongly cospectral with respect to $M(X)$. For every eigenvalue $\mu$ of $\textbf{A}(X \square Y)$ (resp., $\mathcal{A}(X\times Y)$), let $\Lambda_{\mu}$ be the set of all $\lambda\in\sigma_u(\textbf{A}(X))$ (resp., $\mathcal{A}(X)$) and $\Theta_{\mu}$ be the set of all $\theta\in\sigma_w(\textbf{A}(Y))$ (resp., $\mathcal{A}(X)$) such that $\mu+\alpha=\lambda+\theta$ (resp., $\mu=\lambda\theta$).
\begin{enumerate}
\item Let $M(X)=\textbf{A}(X)$. For any vertex $w$ in $Y$, $(u,w)$ and $(v,w)$ are strongly cospectral with respect to $\textbf{A}(X\square Y)$ if and only if for any eigenvalue $\mu$ of $\textbf{A}(X\square Y)$, one of the following holds.
\begin{enumerate}
\item $\mu+\alpha=\lambda+\theta\neq \lambda'+\theta'$ for any $\lambda'\in \sigma(\textbf{A}(X))$ and $\theta'\in \sigma(\textbf{A}(Y))$ such that $\lambda\neq \lambda'$ and $\theta\neq \theta'$.
\item Either $E_{\lambda}\textbf{e}_u=E_{\lambda}\textbf{e}_v$ for all $\lambda\in\Lambda_{\mu}$ or $E_{\lambda}\textbf{e}_u=-E_{\lambda}\textbf{e}_v$ for all $\lambda\in\Lambda_{\mu}$.
\end{enumerate}
\item Let $M(X)=\mathcal{A}(X)$ and $X$ and $Y$ be simple. For any vertex $w$ in $Y$, $(u,w)$ and $(v,w)$ are strongly cospectral with respect to $\mathcal{A}(X\times Y)$ if and only if for any eigenvalue $\mu$ of $\mathcal{A}(X\times Y)$, either
\begin{enumerate}
\item $\frac{1}{\gamma}\mu-\alpha\gamma=(\lambda-\alpha)(\theta-\alpha)\neq (\lambda'-\alpha)(\theta'-\alpha)$ for any $\lambda'\in \sigma(\textbf{A}(X))$ and $\theta'\in \sigma(\textbf{A}(Y))$ such that $\lambda\neq \lambda'$ and $\theta\neq \theta'$, or
\item either $E_{\lambda}\textbf{e}_u=E_{\lambda}\textbf{e}_v$ for all $\lambda\in\Lambda_{\mu}$ or $E_{\lambda}\textbf{e}_u=-E_{\lambda}\textbf{e}_v$ for all $\lambda\in\Lambda_{\mu}$.
\end{enumerate}
\item If we add that $w$ and $z$ are strongly cospectral with respect to $M(Y)$, then $(u,w)$ and $(v,z)$ are strongly cospectral with respect to $\textbf{A}(X\square Y)$ (resp., $\mathcal{A}(X\times Y)$) if and only if for every eigenvalue $\mu$ of $\textbf{A}(X\square Y)$ (resp., $\mathcal{A}(X\times Y)$), either 
\begin{enumerate}
\item $E_{\lambda}\textbf{e}_u\otimes E_{\theta}\textbf{e}_w=E_{\lambda}\textbf{e}_v\otimes E_{\lambda}\textbf{e}_z$ for all $\lambda\in\Lambda_{\mu}$ and $\theta\in\Theta_{\mu}$, or
\item $E_{\lambda}\textbf{e}_u\otimes E_{\theta}\textbf{e}_w=-E_{\lambda}\textbf{e}_v\otimes E_{\lambda}\textbf{e}_z$ for all $\lambda\in\Lambda_{\mu}$ and $\theta\in\Theta_{\mu}$.
\end{enumerate}
\end{enumerate}
\end{theorem}

\begin{proof}
Suppose $u$ and $v$ are strongly cospectral with respect to $\textbf{A}(X)$, and let $w$ be any vertex in $Y$. Note that the characteristic vector of $(u,w)$ is $\textbf{e}_u\otimes \textbf{e}_w$. We first prove 1. By (\ref{Cart}), we have
\begin{equation}
\label{cart}
\begin{split}
\textbf{A}(X\square Y)+\alpha I &=2\alpha I+\beta\left(D(X)\otimes I\right)+\beta\left(I\otimes D(Y)\right)+\gamma \left(A(X)\otimes I\right)+\gamma \left(I\otimes A(Y)\right)\\
&=\left((\alpha I+\beta D(X)+\gamma A(X))\otimes I\right)+\left(I\otimes (\alpha I+\beta D(Y)+\gamma A(Y))\right)\\
&=\left(\textbf{A}(X)\otimes I\right)+\left(I\otimes\textbf{A}(Y)\right).
\end{split}
\end{equation}
Let $\lambda$ and $\theta$ are eigenvalues of $\textbf{A}(X)$ and $\textbf{A}(Y)$ with associated eigenvectors $\textbf{v}$ and $\textbf{w}$. Then (\ref{cart}) yields
\begin{equation}
\label{alq}
\begin{split}
(\textbf{A}(X\square Y)+\alpha I)(\textbf{v}\otimes \textbf{w})&=(\textbf{A}(X)\otimes I)\textbf{v}+(I\otimes\textbf{A}(Y))\textbf{w}=(\lambda+\theta)(\textbf{v}\otimes \textbf{w}).
\end{split}
\end{equation}
For any eigenvalue $\mu$ of $\textbf{A}(X\square Y)$, one checks that
\begin{equation}
\label{koda}
E_{\mu}=\sum_{\substack{\lambda\in \textbf{A}(X),\ \theta\in \textbf{A}(Y)\\ \lambda+\theta=\mu+\alpha}}(E_{\lambda}\otimes E_{\theta}).
\end{equation}
We proceed with two cases, when $E_{\mu}$ in (\ref{koda}) consists of only one or at least two summands. Consider the sets $\Lambda_{\mu}$ and $\Theta_{\mu}$ defined in the premise. The first case is equivalent to the fact that $\mu+\alpha=\lambda+\theta\neq \lambda'+\theta'$ for any $\lambda'\in \sigma(\textbf{A}(X))$ and $\theta'\in \sigma(\textbf{A}(Y))$ such that $\lambda\neq \lambda'$ or $\theta\neq \theta'$, and it follows that
\begin{equation*}
\begin{split}
E_{\mu}(\textbf{e}_u\otimes \textbf{e}_w)=E_{\lambda}\textbf{e}_u\otimes E_{\theta}\textbf{e}_w=(\pm E_{\lambda}\textbf{e}_v)\otimes E_{\theta}\textbf{e}_w=\pm (E_{\lambda}\textbf{e}_v\otimes E_{\theta}\textbf{e}_w)=\pm E_{\mu}(\textbf{e}_v\otimes \textbf{e}_w),
\end{split}
\end{equation*}
On the other hand, if $E_{\mu}$ consists of at least two summands, then
\begin{equation}
\label{koda1}
\begin{split}
E_{\mu}(\textbf{e}_u\otimes \textbf{e}_w)&=\sum_{\lambda\in\Lambda_{\mu},\ \theta\in\Theta_{\mu}}(E_{\lambda}\textbf{e}_u\otimes E_{\theta}\textbf{e}_w)\stackrel{(\star)}{=}\pm \sum_{\lambda\in\Lambda_{\mu},\ \theta\in\Theta_{\mu}}( E_{\lambda}\textbf{e}_v\otimes  E_{\theta}\textbf{e}_w)=\pm E_{\mu}(\textbf{e}_v\otimes \textbf{e}_w),
\end{split}
\end{equation}
where $(\star)$ holds if and only if $E_{\lambda}\textbf{e}_u=E_{\lambda}\textbf{e}_v$ for all $\lambda\in\Lambda_{\mu}$ or $E_{\lambda}\textbf{e}_u=-E_{\lambda}\textbf{e}_v$ for all $\lambda\in\Lambda_{\mu}$. This proves 1. Next, suppose $X$ and $Y$ are simple. Since (\ref{weak}) holds, we have
\begin{equation*}
\begin{split}
\gamma D(X\times Y)^{-\frac{1}{2}}A(X\times Y)D(X\times Y)^{-\frac{1}{2}}&=\gamma(D(X)^{-\frac{1}{2}}\otimes D(Y)^{-\frac{1}{2}})(A(X)\otimes A(Y))(D(X)^{-\frac{1}{2}}\otimes D(Y)^{-\frac{1}{2}})\\
&=\gamma\left(D(X)^{-\frac{1}{2}}A(X)D(X)^{-\frac{1}{2}}\right)\otimes \left(D(Y)^{-\frac{1}{2}}A(Y)D(Y)^{-\frac{1}{2}}\right)\\
&=\frac{1}{\gamma}\left((\mathcal{A}(X)-\alpha I)\otimes (\mathcal{A}(Y)-\alpha I)\right).
\end{split}
\end{equation*}
Consequently,
\begin{equation}
\label{bnb}
\gamma\mathcal{A}(X\times Y)-\alpha \gamma I=(\mathcal{A}(X)-\alpha I)\otimes(\mathcal{A}(Y)-\alpha I)
\end{equation}
If $\lambda$ and $\theta$ are eigenvalues of $\mathcal{A}(X)$ and $\mathcal{A}(Y)$ with corresponding eigenvectors $\textbf{v}$ and $\textbf{w}$, then (\ref{bnb}) yields
\begin{equation}
\label{nl}
\begin{split}
(\gamma\mathcal{A}(X\times Y)-\alpha\gamma I)(\textbf{v}\otimes \textbf{w})&=(\mathcal{A}(X)-\alpha I)\textbf{v}\otimes(\mathcal{A}(Y)-\alpha I)\textbf{w}=(\lambda-\alpha)(\theta-\alpha)(\textbf{v}\otimes \textbf{w}).
\end{split}
\end{equation}
Applying the same argument from the proof of 1 establishes 2. Finally, if $w$ and $z$ are strongly cospectral with respect to $M(Y)$, then a calculation similar to (\ref{koda1}) implies that $(u,w)$ and $(v,z)$ are strongly cospectral with respect to $\textbf{A}(X\square Y)$ (resp., $\mathcal{A}(X\times Y)$) if and only if for every eigenvalue $\mu$ of $\textbf{A}(X\square Y)$ (resp., $\mathcal{A}(X\times Y))$, either (i) $E_{\lambda}\textbf{e}_u\otimes E_{\theta}\textbf{e}_w=E_{\lambda}\textbf{e}_v\otimes E_{\lambda}\textbf{e}_z$ for all $\lambda\in\Lambda_{\mu}$ and $\theta\in\Theta_{\mu}$, in which case $E_{\mu}(\textbf{e}_u\otimes\textbf{e}_w)=E_{\mu}(\textbf{e}_v\otimes\textbf{e}_z)$, or (ii) $E_{\lambda}\textbf{e}_u\otimes E_{\theta}\textbf{e}_w=-E_{\lambda}\textbf{e}_v\otimes E_{\lambda}\textbf{e}_z$ for all $\lambda\in\Lambda_{\mu}$ and $\theta\in\Theta_{\mu}$, in which case $E_{\mu}(\textbf{e}_u\otimes\textbf{e}_w)=-E_{\mu}(\textbf{e}_v\otimes\textbf{e}_z)$.
\end{proof}

For $M=A,L,Q$, (\ref{Cart}) implies that $M(X\square Y)=M(X)\otimes I+I\otimes M(Y)$, and (\ref{alq}) implies the eigenvalues of $M(X\square Y)$ are $\lambda+\theta$, where $\lambda$ and $\theta$ are the eigenvalues of $M(X)$ and $M(Y)$, respectively. For $\mathcal{A}=\mathcal{L}$, (\ref{bnb}) implies that $\mathcal{L}(X\times Y)=\mathcal{L}(X)\otimes I+I\otimes \mathcal{L}(Y)-(\mathcal{L}(X)\otimes \mathcal{L}(Y))$ and (\ref{nl}) implies that the eigenvalues of $\mathcal{L}(X\times Y)$ are $\lambda+\theta-\lambda\theta$, where $\lambda$ and $\theta$ are the eigenvalues of $\mathcal{L}(X)$ and $\mathcal{L}(Y)$, respectively.

We also observe that $u$ and $v$ are strongly cospectral vertices with respect to $M(X)$, and every eigenvalue $\mu$ of $\textbf{A}(X \square Y)$ (resp., $\mathcal{A}(X\times Y)$) is simple, then it follows from Theorem \ref{closed}(1a) (resp., (1b)) that $(u,w)$ and $(y,w)$ are strongly cospectral with respect to $\textbf{A}(X \square Y)$ (resp., $\mathcal{A}(X\times Y)$). In particular, one checks that the eigenvalues of  $\textbf{A}(X \square Y)$ (resp., $\mathcal{A}(X\times Y)$) are simple if and only if the eigenvalues of $\textbf{A}(X)$ and $\textbf{A}(Y)$ are all simple and the sums $\lambda+\theta$ (resp., $\lambda\theta-\alpha(\lambda+\theta)$) are unique for each $\lambda\in \textbf{A}(X)$ and $\theta\in \textbf{A}(Y)$. 

To illustrate Theorem \ref{closed}, we give the following examples.

\begin{example}
Let $x$ and $y$ be the vertices of $K_2$, and $u$ and $v$ be the antipodal vertices of $P_3$. One checks that $x$ and $y$ are strongly cospectral with respect to $\textbf{A}(K_2)$, and $u$ and $v$ are strongly cospectral with respect to $\textbf{A}(P_3)$. The eigenvalues of $\textbf{A}(K_2)$ are $\lambda_1=\alpha+\beta+\gamma$ and $\lambda_2=\alpha+\beta-\gamma$, while the eigenvalues of $\textbf{A}(P_3)$ are $\theta_1=\alpha+\beta$, $\theta_2=\alpha+\frac{1}{2}(3\beta+\sqrt{\beta^2+8\gamma^2})$ and $\theta_3=\alpha+\frac{1}{2}(3\beta-\sqrt{\beta^2+8\gamma^2})$. We have two cases.
\begin{itemize}
\item Let $\beta\neq \pm 8\gamma$. Then the eigenvalues of $\textbf{A}(K_2\square P_3)$ are $\theta_i+\lambda_j-\alpha$ for $i\in\{1,2\}$ and $j\in\{1,2,3\}$. Invoking Theorem \ref{closed}(1a) yields four pairwise strongly cospectral vertices $(x,u)$, $(y,u)$, $(x,v)$ and $(y,v)$ with respect to $\textbf{A}(K_2\square P_3)$ (see Figure \ref{KP}).
\item Let $\beta=8\gamma$. The eigenvalues of $\textbf{A}(K_2\square P_3)$ are $\lambda_1+\theta_1-\alpha$ (multiplicity two), $\lambda_1+\theta_2-\alpha$, $\lambda_1+\theta_3-\alpha$, $\lambda_2+\theta_1-\alpha$, and  $\lambda_2+\theta-\alpha$, where $\theta=\theta_3$ whenever $\gamma<0$ and $\theta=\theta_2-\alpha$ whenever $\gamma>0$. One checks that $E_{\lambda_1+\theta_1-\alpha}=(E_{\lambda_1}\otimes E_{\theta_1})+(E_{\lambda_2}\otimes E_{\theta})$. Since $E_{\lambda_1}\textbf{e}_x=E_{\lambda_1}\textbf{e}_y$ and $E_{\lambda_2}\textbf{e}_x=-E_{\lambda_2}\textbf{e}_y$, we get $E_{\lambda_1+\theta_1-\alpha}(\textbf{e}_x\otimes \textbf{e}_u) \neq \pm E_{\lambda_1+\theta_1-\alpha}(\textbf{e}_x\otimes \textbf{e}_u)$. Moreover, since $E_{\theta} \textbf{e}_u=E_{\theta} \textbf{e}_v$ and $E_{\theta_1} \textbf{e}_u=-E_{\theta_1} \textbf{e}_v$, we also get $E_{\lambda_1+\theta_1-\alpha}(\textbf{e}_x\otimes \textbf{e}_u) \neq \pm E_{\lambda_1+\theta_1-\alpha}(\textbf{e}_x\otimes \textbf{e}_v)$. Thus, none of $(x,u)$, $(y,u)$, $(x,v)$ and $(y,v)$ are pairwise strongly cospectral with respect to $\textbf{A}(K_2\square P_3)$. The same holds whenever $\beta=-8\gamma$.
\end{itemize}
\end{example}

\begin{figure}[h!]
	\begin{center}
		\begin{tikzpicture}
		\tikzset{enclosed/.style={draw, circle, inner sep=0pt, minimum size=.25cm}}
	   
		\node[enclosed,fill=cyan,label={above, yshift=0cm: $(x,u)$}] (v_1) at (0,1.5) {};
		\node[enclosed] (v_2) at (1.5,1.5) {};
		\node[enclosed,fill=cyan,label={above, yshift=0cm: $(x,v)$}] (v_3) at (3,1.5) {};
		\node[enclosed,fill=cyan,label={below, yshift=0cm: $(y,u)$}] (v_4) at (0,0) {};
		\node[enclosed] (v_5) at (1.5,0) {};
		\node[enclosed,fill=cyan, label={below, yshift=0cm: $(y,v)$}] (v_6) at (3,0) {};
		
		\draw (v_1) -- (v_2);
		\draw (v_1) -- (v_4);
 		\draw (v_2) -- (v_3);
 		\draw (v_2) -- (v_5);
 		\draw (v_4) -- (v_5);
 		\draw (v_3) -- (v_6);
 		\draw (v_5) -- (v_6);
 		
 		\node[enclosed,fill=cyan,label={above, yshift=0cm: $(x,u)$}] (v_1) at (6,1.5) {};
		\node[enclosed,label={above, yshift=0cm: $(x,w)$}] (v_2) at (7.5,1.5) {};
		\node[enclosed,fill=magenta,label={above, yshift=0cm: $(x,w)$}] (v_3) at (9,1.5) {};
		\node[enclosed,fill=cyan,label={below, yshift=0cm: $(y,u)$}] (v_4) at (6,0) {};
		\node[enclosed, label={below, yshift=0cm: $(y,v)$}] (v_5) at (7.5,0) {};
		\node[enclosed,fill=magenta, label={below, yshift=0cm: $(y,w)$}] (v_6) at (9,0) {};
		
		\draw (v_1) -- (v_5);
		\draw (v_2) -- (v_4);
 		\draw (v_5) -- (v_3);
 		\draw (v_2) -- (v_6);
 		\draw (v_1) -- (v_6);
 		\draw (v_3) -- (v_4);
 		
		\end{tikzpicture}
	\end{center}
	\caption{$K_2\square P_3$ with four pairwise strongly cospectral vertices with respect to $\textbf{A}$ whenever $\beta\neq \pm 8\gamma$ marked blue (left) and $K_2\times K_3$ with pairs of strongly cospectral vertices with respect to $\mathcal{A}$ marked blue, white, and pink (right)}\label{KP}
\end{figure}
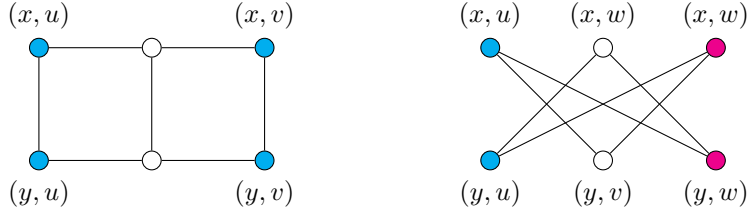

\begin{example}
Consider $K_2\times K_3$ with $V(K_2)=\{x,y\}$ and $V(K_3)=\{u,v,w\}$. As $\mathcal{A}(K_2)=\alpha I+\gamma A(K_2)$ and $\mathcal{A}(K_3)=\alpha I+\frac{1}{2}\gamma A(K_3)$, $x$ and $y$ are strongly cospectral with respect to $\mathcal{A}(K_2)$, and the eigenvalues of $\mathcal{A}(K_2)$ and $\mathcal{A}(K_3)$ are $\alpha\pm \gamma$ and $\alpha+\gamma,\alpha-\frac{1}{2}\gamma$ (multiplicity two), respectively. From (\ref{nl}), the eigenvalues of $\mathcal{A}(K_2\times K_3)$ are $\pm \gamma^2$ and $\pm \frac{1}{2}\gamma^2$ (with multiplicity two). Applying Theorem \ref{closed}(2a), we get that $(x,z)$ and $(y,z)$ are strongly cospectral with respect to $\mathcal{A}(K_2\times K_3)$ for any $z\in V(K_3)$ (see Figure \ref{KP}). 
\end{example}

%%%%%%%%%%%%%%%%%%%%%%%%%%%%%%%%%%%%%%%%%%%

\section{Equitable and almost equitable partitions}\label{eqtp}

Let $\pi=(C_1,\ldots,C_n)$ be a partition of $V(X)$. The \textit{normalized characteristic matrix} $\mathcal{P}$ of $\pi$ is the $|V(X)|\times n$ matrix such that $(\mathcal{P})_{j,\ell}=1/\sqrt{|C_{\ell}|}$ if $j\in C_{\ell}$ and $(\mathcal{P})_{j,\ell}=0$ otherwise. It is known that $\mathcal{P}^T\mathcal{P}=I_{|V(X)|}$ and $\mathcal{P}\mathcal{P}^T$ is a $|V(X)|\times |V(X)|$ block diagonal matrix consisting of $|C_j|\times |C_j|$ diagonal blocks of the form $\frac{1}{|C_j|}J$. We say that $\pi$ is \textit{equitable} if for every $j,\ell\in \{1,\ldots,n\}$, the sum of the weights of the edges $(u,v)$ for a fixed $u\in C_j$ and for $v\in C_{\ell}$ is a constant $d_{j,\ell}$. Consequently, $|C_j|d_{j,\ell}=|C_{\ell}|d_{\ell,j}$ for every $j,\ell\in \{1,\ldots,n\}$, and in the subgraph induced by each $C_j$, the sum of the weights of all edges incident to each vertex in $C_j$ is equal. For simple graphs, the latter implies that the subgraph induced by $C_j$ is a weighted $d_{j,j}$-regular graph. If we only require that for every $j,\ell\in \{1,\ldots,n\}$ with $j\neq \ell$, the sum of the weights of the edges $(u,v)$ for a fixed $u\in C_j$ and for $v\in C_{\ell}$ is a constant $d_{j,\ell}$, then we say that $\pi$ is an \textit{almost equitable partition} of $V(X)$. An equitable partition is almost equitable, but the converse is not true. Moreover, every graph admits an equitable and almost equitable partition. The partition of $V(X)$ whose cells consist of a single vertex is both almost equitable and equitable, while the partition with a single cell consisting all the vertices in $V(X)$ is almost equitable but not equitable, unless $X$ is weighted regular. These two are called \textit{trivial} equitable and almost equitable partitions. An (almost) equitable partition $\pi$ gives rise to a \textit{quotient graph} $X/\pi$ whose vertices are the cells of $\pi$ and the weight of the edge joining $C_j$ and $C_{\ell}$ is given by $\sqrt{d_{j,\ell}{d_{\ell,j}}}$. For the basics of equitable and almost equitable partitions, see \cite{CARDOSO,Godsil:AlgebraicGraph}. Following the treatment in \cite{Alvir2016,CARDOSO,Godsil2012a}, we provide a characterization of equitable partitions and almost equitable partitions in weighted graphs with possible loops.

\begin{theorem}
\label{ep}
Let $\pi$ be a partition of $V(X)$. If $\beta\neq -\gamma$, then the following statements are equivalent.
\begin{enumerate}
\item $\pi$ is equitable with $n$ cells.
\item  \label{lap2}The column space of $\mathcal{P}$ is $\textbf{A}(X)$-invariant.
\item \label{lap} There is a $n\times n$ matrix $B$ such that $\textbf{A}(X)\mathcal{P}=\mathcal{P}B$.
\item  \label{lap3} $\textbf{A}(X)$ and $\mathcal{P}\mathcal{P}^T$ commute.
\end{enumerate}
On the other hand, if $\beta=-\gamma$, then $\pi$ is almost equitable if and only if  statements \ref{lap2}, \ref{lap} or \ref{lap3} hold. In addition, if either $\pi$ is equitable or $\pi$ is almost equitable and $\beta=-\gamma$, then $B=\mathcal{M}(X/\pi)$, where
\begin{equation}
\label{eq}
(\mathcal{M}(X/\pi))_{j,\ell}= 
\begin{cases}
    \gamma\sqrt{d_{j,\ell}d_{\ell,j}}, & \text{if}\ j\neq \ell\\
    \alpha+(\beta+\gamma)d_{j,j}+\beta\displaystyle\left(\sum_{r\neq j} d_{j,r}+\frac{1}{|C_j|}\sum_{u\in C_j}(A(X))_{u,u}\right), & \text{if}\ j=\ell.
\end{cases}
\end{equation}
\end{theorem}

We remark that $\mathcal{M}(X/\pi)\neq \textbf{A}(X/\pi)$. Moreover, a simple generalization of \cite[Theorem 9.3.3]{Godsil:AlgebraicGraph} yields $\sigma(\mathcal{M}(X/\pi))\subseteq\sigma(\textbf{A}(X))$.

Let $\pi$ be an equitable partition of $V(X)$ with characteristic matrix $\mathcal{P}$, and $\textbf{A}(X)=\sum_j\theta_jE_j$ be a spectral decomposition of $\textbf{A}(X)$. Using Theorem \ref{ep}(\ref{lap}) and the properties of $\mathcal{P}$ and the $E_j$'s, $\mathcal{M}(X/\pi)=\mathcal{P}^T\textbf{A}(X)P=\sum_j\theta_j\left(\mathcal{P}^TE_j \mathcal{P}\right)$ is a spectral decomposition of $\mathcal{M}(X/\pi)$. Let $C_{j}=\{u\}$ be a singleton cell in $\pi$ so that $\mathcal{P}\textbf{e}_{\{u\}}=\textbf{e}_u$. Following the same argument in \cite[Lemma 3.2]{Fan2013}, we get the following.

\begin{theorem}
\label{epsc}
Let $\pi$ be a partition of $V(X)$ with singleton cells $\{u\}$ and $\{v\}$. If either $\pi$ is equitable or $\pi$ is almost equitable and $\beta=-\gamma$, then $u$ and $v$ are strongly cospectral with respect to $\textbf{A}(X)$ if and only if $\{u\}$ and $\{v\}$ are strongly cospectral with respect to $\mathcal{M}(X/\pi)$.
\end{theorem}

As $\textbf{e}_{\{u\}}^Te^{it\mathcal{M}(X/\pi)}\textbf{e}_{\{v\}}= \textbf{e}_{\{u\}}^Te^{it(\mathcal{P}^T\textbf{A}(X)P)}\textbf{e}_{\{v\}}=(\mathcal{P}\textbf{e}_{\{u\}})^Te^{it\textbf{A}(X)}P\textbf{e}_{\{v\}} =\textbf{e}_u^Te^{it\textbf{A}(X)}\textbf{e}_v$ because the exponential function is analytic, we obtain an extension of a result of Bachman et al.\ in \cite[Theorem 2]{Bachman2011}.

\begin{theorem}
\label{bachman}
Let $\pi$ be a partition of $V(X)$ with singleton cells $\{u\}$ and $\{v\}$. If either $\pi$ is equitable, or $\pi$ is almost equitable and $\beta=-\gamma$, then $\left(e^{it\mathcal{M}(X/\pi)}\right)_{\{u\},\{v\}}=\left(e^{it\textbf{A}(X)}\right)_{u,v}$.
\end{theorem}

Lastly, we give the following result about twin vertices that are singleton cells in an equitable partition.

\begin{theorem}
\label{twineq}
Let $\pi$ of $V(X)$ with singleton cells $\{u\}$ and $\{v\}$ such that either $\pi$ is equitable, or $\pi$ is almost equitable and $\beta=-\gamma$. If $u,v\in T(\omega,\eta)$, then $\textbf{e}_{\{u\}}-\textbf{e}_{\{v\}}$ is an eigenvector for $\mathcal{M}(X/\pi)$ associated to $\theta$ given in (\ref{eval}). Conversely, if $\textbf{e}_{\{u\}}-\textbf{e}_{\{v\}}$ is an eigenvector for $\mathcal{M}(X/\pi)$ associated to $\theta$ given in (\ref{eval}), and either $\beta\neq -\frac{\gamma}{2}$ or $(A(X))_{u,u}=(A(X))_{v,v}$, then $u,v\in T(\omega,\eta)$.
\end{theorem}

\begin{proof}
The forward implication is clear, so it suffices to prove the converse. Assume $\pi=(C_1,C_2,\ldots,C_n)$, where $C_1=\{u\}$ and $C_2=\{v\}$. Suppose $\textbf{e}_1-\textbf{e}_2$ is an eigenvector for $\mathcal{M}(X/\pi)$ associated to $\theta$ given in (\ref{eval}). That is, $\mathcal{M}(X/\pi)\left(\textbf{e}_1-\textbf{e}_2\right)=\theta\left(\textbf{e}_1-\textbf{e}_2\right)$. Comparing entries yields $\mathcal{M}(X/\pi)_{1,1}=\mathcal{M}(X/\pi)_{2,2}$ and $\mathcal{M}(X/\pi)_{j,1}=\mathcal{M}(X/\pi)_{j,2}$ for each $j\neq u,v$. The latter equation yields $d_{j,1}d_{1,j}=d_{j,2}d_{2,j}$ for $j\neq 1,2$. Since $d_{1,j}=|C_j|d_{j,1}$ and $d_{2,j}=|C_j|d_{j,2}$, we obtain $d_{j,1}=d_{j,2}$. That is, $(w,u)$ and $(w,v)$ have equal equal weights for any $w\neq u,v$. If $\beta\neq -\frac{\gamma}{2}$, the former equation yields $(A)_{u,u}=(A)_{v,v}=\omega$. Thus, if we add that either $\beta\neq -\frac{\gamma}{2}$ or $(A)_{u,u}=(A)_{v,v}=\omega$, then $u,v\in T(\omega,\eta)$, where $\eta$ is the weight of the edge between $u$ and $v$.
\end{proof}

%%%%%%%%%%%%%%%%%%%%%%%%%%%%%%%%%%%%%%%

\section{Joins}\label{joins}

Let $\omega,\eta\in\mathbb{R}$. For $n\geq 1$, let $\textbf{O}_n(\omega)$ denote the \textit{empty graph} on $n$ vertices with possible loops, where every loop on each vertex has weight $\omega$ and for $n\geq 2$, let $\textbf{K}_n(\omega,\eta)$ denote the \textit{weighted complete graph} on $n$ vertices with possible loops, where every loop on each vertex has weight $\omega$, and every edge between distinct vertices has weight $\eta\neq 0$. If $\omega=0$, then the loops on these graphs are absent. We also let $ \bigvee_j \textbf{K}_{n_j}(\omega_j,\eta_j)$ be the weighted complete graph with possible loops, where every pair of vertices in $\textbf{K}_{n_j}(\omega_j,\eta_j)$ are true twins in $\bigvee_j \textbf{K}_{n_j}(\omega_j,\eta_j)$ for each $j$ whenever $m_j\geq 2$, and $\bigvee_j \textbf{O}_{m_j}(\omega_j)$ be a weighted complete multipartite graph with possible loops, where every pair of vertices in $\textbf{O}_{m_j}(\omega_j,\eta_j)$ are false twins in $\bigvee_j \textbf{O}_{m_j}(\omega_j)$ for each $j$ whenever $m_j\geq 2$. In particular, if $j\in\{1,2\}$, then $\textbf{O}_{m_1}(\omega_1)\vee \textbf{O}_{m_2}(\omega_2)$ is a weighted complete bipartite graph with possible loops. We state a basic fact about graphs with twins.

\begin{proposition}
\label{propo}
Let $T=T(\omega,\eta)$ be a set of twins in $X$ with $|T|=m$.
\begin{enumerate}
\item If $\eta\neq 0$, then the induced subgraph of $T$ is isomorphic to $\textbf{K}_{|T|}(\omega,\eta)$. The following also hold.
\begin{enumerate}
\item $T=V(X)$ if and only if $X\cong \textbf{K}_{|V(X)|}(\omega,\eta)$.
\item If $T\neq V(X)$, then $m\leq |V(X)|-2$. Moreover, if $V(X)\backslash T$ is a set of false twins in $X$, then $X\cong \textbf{K}_{m}(\omega,\eta)\vee \textbf{O}_{|V(X)|-m}(\omega')$.
\item  \label{join2} If $\{T_j\}$ is a partition of $V(X)$ such that each $T_j=T_j(\omega_j,\eta_j)$ is a set of true twins in $X$ with $|T_j|=m_j$, then $X\cong\bigvee_j \textbf{K}_{m_j}(\omega_j,\eta_j)$.
\end{enumerate}
\item If $\eta=0$, then the induced subgraph of $T$ is isomorphic to $\textbf{O}_{m}(\omega)$. The following also hold.
\begin{enumerate}
\item $m\leq |V(X)|-1$, and if $m=|V(X)|-1$, then $X\cong \textbf{O}_{m}(\omega)\vee \textbf{O}_{1}(\omega')$.
\item \label{join1} If $\{T_j\}$ is a partition of $V(X)$ such that each $T_j=T_j(\omega_j,0)$ is a set of false twins in $X$ with $|T_j|=m_j$, then $X\cong\bigvee_j \textbf{O}_{m_j}(\omega_j)$.
\end{enumerate}
\end{enumerate}
\end{proposition}

The following example reveals an important fact about $\textbf{K}_n(\omega,\eta)$.

\begin{example}
\label{km}
The two vertices in $\textbf{K}_2(\omega,\eta)$ are strongly cospectral with respect to $M$. Now, let $n\geq 3$. By Corollary \ref{cosptw}, all vertices of $\textbf{K}_n(\omega,\eta)$ are pairwise cospectral with respect to $M$. Since Lemma \ref{eu-ev} implies that $\theta$ is an eigenvalue of $M$ with multiplicity $n-1$, Proposition \ref{esupp} yields $|\sigma_u(M)|=2$ for any vertex $u$. Invoking Theorem \ref{esupp+}, no two vertices of $\textbf{K}_n(\omega,\eta)$ are strongly cospectral with respect to $M$.
\end{example}

It turns out, the same holds for joins of $\textbf{K}_n(\omega,\eta)$ and joins of $\textbf{O}_m(\omega)$.

\begin{example}
\label{threetw1}
Let $X\cong \bigvee _j \textbf{K}_{n_j}(\omega_j,\eta_j)$ and $Y\cong \bigvee _j \textbf{O}_{m_j}(\omega_j)$. Assume $\{T_j\}$ is a partition of the vertex set of $X$ and $Y$ such that the induced subgraph of each $T_j$ in $X$ is $\textbf{K}_{n_j}(\omega_j,\eta_j)$ while the induced subgraph of each $T_j$ in $Y$ is $\textbf{O}_{m_j}(\omega_j)$. Every pair of vertices in $T_j$ are true twins in $X$ and false twins in $Y$. If $|T_j|\geq 3$, then, by Corollary \ref{3tw}, any vertex in $T_j$ cannot be strongly cospectral with any vertex of $T_{\ell}$ for $j\neq \ell$ with respect to $M$. In particular, if $|T_j|\geq 3$ for each $j$, then no two vertices in $X$ and $Y$ are strongly cospectral. These results extend to the join $X\vee Y$ provided every pair of vertices in each $T_j$ are twins in $X\vee Y$.
\end{example}

We now investigate strong cospectrality in joins of the complete or empty graph with a regular graph.

\begin{theorem}
\label{dc2}
For $n\geq 2$, let $X$ be either $\textbf{K}_n(\omega,\eta)$ or $\textbf{O}_n(\omega)$, and for $n=1$, let $X\cong \textbf{O}_1(\omega)$. For any weighted graph $H$ on $m$ vertices with possible loops, define $X\vee H$ as the connected weighted graph such that for a fixed $v\in V(H)$, the edges $(u,v)$ for each $u\in V(X)$ have the equal weights. The following hold.
\begin{enumerate}
\item \label{dc3} For $n\geq 3$, no vertex in $X$ is strongly cospectral with any vertex in $X\vee H$ with respect to $M(X\vee H)$.
%In particular, no pair of vertices in $\textbf{O}_{n}(\omega)\vee \textbf{O}_{1}(\omega')$ are strongly cospectral with respect to $M(X)$.
\item \label{wawa1} Suppose $n=1$, and $V(X)=\{u\}$. If
\begin{equation}
\label{dc}
\operatorname{tr}(\textbf{A}(X\vee H)[u])=\operatorname{tr}(\textbf{A}(X\vee H)[v])
\end{equation}
does not hold for some $v\in V(H)$, then $u$ and $v$ are not strongly cospectral with respect to $\textbf{A}(X\vee H)$.
\item \label{yes} Let $n=2$. Assume each edge joining $X$ and $H$ has weight $\delta$, and let $\pi=(C_1,C_2,C_3)$ be a partition of $V(X\vee H)$ such that $C_1$ and $C_2$ are singleton cells containing each of the two vertices in $X$ and $C_3=V(H)$. Then $\pi$ is almost equitable, and if we further assume that $\operatorname{deg}(w)-(A)_{w,w}$ is a constant $d$ for each $w\in V(H)$, then $\pi$ is equitable. Moreover, if either $\beta=-\gamma$ or $\operatorname{deg}(w)-(A)_{w,w}$ is a constant $d$ for each $w\in V(H)$, then $u$ and $v$ are strongly cospectral with respect to $\textbf{A}(X\vee H)$ if and only if the entries of $\mathcal{M}=\mathcal{M}(X\vee H/\pi)$ do not satisfy
\begin{equation}
\label{M}
[(\mathcal{M})_{1,3}]^2 =(\mathcal{M})_{1,2}[(\mathcal{M})_{1,2}-(\mathcal{M})_{1,1}+(\mathcal{M})_{3,3}].
%(\mathcal{M}_{1,1}+\mathcal{M}_{1,2}-\mathcal{M}_{3,3})^2+8(\mathcal{M}_{1,3})^2 =(3\mathcal{M}_{1,2}+\mathcal{M}_{3,3}-\mathcal{M}_{1,1})^2.
\end{equation}
\end{enumerate}
\end{theorem}

\begin{figure}[h!]
	\begin{center}
		\begin{tikzpicture}
		\tikzset{enclosed/.style={draw, circle, inner sep=0pt, minimum size=.2cm}}
	   
	   \node[enclosed, fill=cyan] (u_1) at (2,2) {};
		\node[enclosed, fill=cyan] (u_2) at (1.3,1.35) {};
		\node[enclosed, fill=cyan] (u_3) at (2.7,1.35) {};
		\node[enclosed] (u_4) at (2,-0.05) {};
		\node[enclosed] (u_5) at (1,0.25) {};
		\node[enclosed] (u_6) at (3,0.25) {};
		
		\draw (u_1) --  (u_2);
		\draw (u_2) --  (u_3);
		\draw (u_1) --  (u_3);
		\draw (u_1) --  (u_4);
		\draw (u_1) --  (u_5);
		\draw (u_1) --  (u_6);
		\draw (u_2) --  (u_4);
		\draw (u_2) --  (u_5);
		\draw (u_2) --  (u_6);
		\draw (u_3) --  (u_4);
		\draw (u_3) --  (u_5);
		\draw (u_3) --  (u_6);
		
		\node[enclosed, fill=cyan] (v_1) at (5.5,1.8) {};
		\node[enclosed, fill=cyan] (v_2) at (6.5,1.8) {};
		\node[enclosed, fill=cyan] (v_3) at (7.5,1.8) {};
		\node[enclosed] (v_4) at (5,0.3) {};
		\node[enclosed] (v_5) at (6,0.3) {};
		\node[enclosed] (v_6) at (7.,0.3) {};
		\node[enclosed] (v_7) at (8,0.3) {};
			
		\draw (v_1) --  (v_4);
		\draw (v_1) --  (v_5);
		\draw (v_1) --  (v_6);
		\draw (v_1) --  (v_7);
		\draw (v_2) --  (v_4);
		\draw (v_2) --  (v_5);
		\draw (v_2) --  (v_6);
		\draw (v_2) --  (v_7);
		\draw (v_3) --  (v_4);
		\draw (v_3) --  (v_5);
		\draw (v_3) --  (v_6);
		\draw (v_3) --  (v_7);
			
		\end{tikzpicture}
	\end{center}
	\caption{Unweighted joins: $K_3\vee O_3$ (left) and $O_3\vee O_4$ (right)}\label{join}
\end{figure}
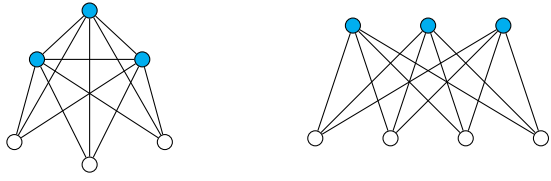

\begin{proof}
For $n\geq 2$, the assumption implies that each pair of vertices in $X$ are twins in $X\vee H$. Thus, if $n\geq 3$, then Corollary \ref{3tw} and Theorem \ref{partw}(\ref{notparr}) altogether proves 1.

Now, let $n=1$ and assume $V(X)=\{u\}$. If (\ref{dc}) does not hold for some $v\in V(H)$, then $\textbf{A}(X\vee H)[u]$ and $\textbf{A}(X\vee H)[v]$ have distinct spectra, and so $\phi_u(X\vee H,t)\neq \phi_v(X\vee H,t)$. That is, $u$ and $v$ are not cospectral with respect to $\textbf{A}$, and in turn, not strongly cospectral with respect to $\textbf{A}$. 

Finally, let $n=2$ and suppose every edge joining $X$ and $H$ has weight $\delta$. Let $V(X)=\{u,v\}$, and consider $\pi=(C_1,C_2,C_3)$ determined by $C_1=\{u\}$ $C_2=\{v\}$ and $C_3=\{V(H)\}$. Then $\pi$ is almost equitable, and if we add that $\operatorname{deg}(w)-(A)_{w,w}$ is constant for all $w\in V(H)$, then $\pi$ is equitable. Now, suppose either $\beta=-\gamma$ or $\operatorname{deg}(w)-(A)_{w,w}$ is constant for all $w\in V(H)$. Then $\mathcal{M}=\mathcal{M}(X\vee H/\pi)$ is a $3\times 3$ real symmetric matrix, and so, $\phi(\mathcal{M},t)$ has real roots. Since $u$ and $v$ are twins in $X\vee H$, Theorem \ref{twineq} implies that $\textbf{e}_1-\textbf{e}_2$ is an eigenvector for $\mathcal{M}$ associated to $\theta$ given in (\ref{eval}). If the entries of $\mathcal{M}$ satisfy (\ref{M}), then $\textbf{e}_1-\textbf{e}_3$ is also an eigenvector for $\mathcal{M}$ associated to the eigenvalue $\theta$, which implies that $\{u\}$ and $\{v\}$ are not parallel, and hence strongly cospectral with respect to $\mathcal{M}$. On the other hand, if (\ref{M}) is violated, then we are guaranteed that the two remaining eigenvalues of $\mathcal{M}$ are conjugates with eigenvectors of the form $[\pm c_1,\pm c_2, 1]^T$. Thus, $\{u\}$ and $\{v\}$ are strongly cospectral with respect to $\mathcal{M}$, and applying Theorem \ref{epsc} yields the desired result.
\end{proof}

Consider the joins in Figure \ref{join}. Since all vertices in $K_3$, $O_3$ and $O_4$ are pairwise twins, Theorem \ref{dc2}(\ref{dc3}) implies that no pair of vertices in $K_3\vee O_3$ and $O_3\vee O_4$ are strongly cospectral with respect to $\textbf{A}$.

In Theorem \ref{dc2}, if $n=1$ and $n=2$, then we call $X\vee H$ a \textit{cone} and a \textit{double cone} on $H$, respectively, and we call the vertices of $X$ as the \textit{apexes} of $X\vee H$. In particular, we call $\textbf{K}_2(\omega,\eta)\vee H$ the \textit{connected} double cone on $H$, and $\textbf{O}_2(\omega)\vee H$ the \textit{disconnected} double cone on $H$.

With the assumption in Theorem \ref{dc2}(\ref{wawa1}), we further assume that $H$ is $d$-regular and each edge between $u$ and any $v\in V(H)$ has weight $\delta$. Then one checks that
\begin{equation*}
\operatorname{tr}(\textbf{A}(X\vee H)[u])=\alpha m+\beta m(d+\delta)+\gamma\left(\sum_{w\in V(H)}(A)_{w,w}\right),
\end{equation*}
and a similar calculation yields
\begin{equation*}
\operatorname{tr}(\textbf{A}(X\vee H)[v])=\alpha m+\beta m\delta+\beta(m-1)(d+\delta)+\gamma\left((A)_{u,u}+\sum_{w\in V(H)\backslash\{v\}}(A)_{w,w}\right),
\end{equation*}
for some $v\in V(H)$. Thus, $\operatorname{tr}(\textbf{A}(X\vee H)[u])=\operatorname{tr}(\textbf{A}(X\vee H)[v])$ if and only if
\begin{equation}
\label{wawa}
\beta[d+\delta(1-m)]+\gamma((A)_{v,v}-(A)_{u,u})=0.
\end{equation}

This yields our next result about cones.

\begin{corollary}
\label{cone}
With the assumption in Theorem \ref{dc2}(\ref{wawa1}), we further assume that $H$ is $d$-regular and each edge between $u$ and any $v\in V(H)$ has weight $\delta$. If (\ref{wawa}) does not hold for some $v\in H$, then $u$ and $v$ are not strongly cospectral with respect to $\textbf{A}(X\vee H)$. If we add that $X\vee H$ is simple, the following hold.
\begin{enumerate}
\item \label{cone1} Let $\beta\neq 0$. If $d\neq \delta(m-1)$, then any two vertices in $X\vee H$ are not strongly cospectral with respect to $\textbf{A}(X\vee H)$.
\item \label{cone2} Let $\beta=0$ and $v\in V(H)$. If $\operatorname{deg}(w)\neq \delta(m-1)$ for some $w\in V(X\vee H\backslash{v})$, then $u$ and $v$ are not strongly cospectral with respect to $\textbf{A}(X\vee H)$. In particular, if $H\backslash v$ is $d'$-regular and $d'\neq \delta(m-2)$, then $u$ and $v$ are not strongly cospectral with respect to $\textbf{A}(X\vee H)$.
\end{enumerate}
\end{corollary}

\begin{proof}
A direct application of Theorem \ref{dc2}(\ref{wawa1}) to (\ref{wawa}) yields statement \ref{cone1}. To prove \ref{cone2}, let $\beta=0$ and without loss of generality, assume that $\alpha=0$. Note that $\phi_z(A(X\vee H))=\phi(A((X\vee H)\backslash z))$ for any $z\in V(X\vee H)$. Now, $\phi(A((X\vee H)\backslash v))=\phi(A(X\vee (H\backslash v)$ for $v\in V(H)$ while $\phi(A((X\vee H)\backslash u))=\phi(A(H))$. Observe that $\operatorname{deg}(u)=\delta(m-1)$ in $X\vee H\backslash{v}$. Thus, if $\operatorname{deg}(w)\neq \delta(m-1)$ for some $w\in V(X\vee H\backslash{v})$, then $\phi(A(X\vee (H\backslash v))\neq \phi(A(H))$ because $H$ is $d$-regular. By Proposition \ref{cosp}, we conclude that $u$ and $v$ not strongly cospectral with respect to $A(X\vee H)$. The latter statement in \ref{cone2} follows immediately.
\end{proof}

With the assumption in Corollary \ref{cone}, let us further suppose that $X\vee H$ is simple and unweighted. By Corollary \ref{cone}, whether $\beta=0$ or $\beta\neq 0$, any two vertices of $X\vee H$ do not exhibit strong cospectrality with respect to $\textbf{A}(X\vee H)$ whenever $d\neq m-1$. However, if $d=m-1$, then $X\vee H\cong K_{m+1}$, which also does not exhibit strong cospectrality with respect to $\textbf{A}(X\vee H)$ by Example \ref{km}. This yields the following result.

\begin{corollary}
\label{conesimpunw}
Any two vertices of a simple unweighted cone on a regular graph do not exhibit strong cospectrality with respect to $\textbf{A}(X\vee H)$.
\end{corollary}

With the assumption in Theorem \ref{dc2}(\ref{yes}), we further suppose that $\operatorname{deg}(w)-(A)_{w,w}$ is a constant $d$ for all $w\in V(H)$ so that $\pi$ is equitable. Making use of the entries of $\mathcal{M}$ in (\ref{eq}), we can write (\ref{M}) as
\begin{equation}
\label{scdoubcone}
\eta\left[(\beta+\gamma)(\omega-d)+\beta\left(\eta+(m-2)\delta+\omega-\frac{1}{m}\sum_{w\in H}(A)_{w,w}\right)-\gamma\eta\right]+\gamma\delta^2m=0.
\end{equation}
If $\eta=0$, then $\gamma\delta^2m=0$, which not true. Hence, by Theorem \ref{dc2}(\ref{yes}), the apexes of $X\vee H$ are strongly cospectral with respect to $\textbf{A}$. Now, suppose $\eta\neq 0$. If $X\vee H$ is simple, then $d=\operatorname{deg}(w)$ is constant for all $w\in V(H)$. In other words, $H$ is a weighted $d$-regular graph, and we can write (\ref{scdoubcone}) as
\begin{equation}
\label{scdoubcone2}
-d(\beta+\gamma)+\beta\left(\eta+(m-2)\delta\right) +\gamma\left(\frac{\delta^2m}{\eta}-\eta\right)=0.
\end{equation}
This yields the following corollary to Theorem \ref{dc2}(\ref{yes}).

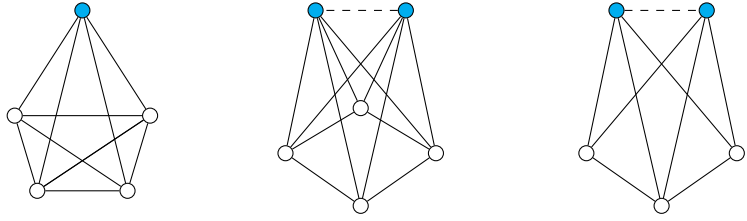
\begin{figure}[h!]
	\begin{center}
		\begin{tikzpicture}
		\tikzset{enclosed/.style={draw, circle, inner sep=0pt, minimum size=.2cm}}
	   
	   \node[enclosed] (u_1) at (-0.1,0.5) {};
		\node[enclosed] (u_2) at (0.2,-0.5) {};
		\node[enclosed] (u_3) at (1.7,0.5) {};
		\node[enclosed] (u_4) at (1.4,-0.5) {};
		\node[enclosed, fill=cyan] (u_5) at (0.8,1.9) {};
		
		\draw (u_1) --  (u_2);
		\draw (u_1) --  (u_3);
		\draw (u_1) --  (u_4);
		\draw (u_2) --  (u_3);
		\draw (u_2) --  (u_3);
		\draw (u_2) --  (u_4);
		\draw (u_3) --  (u_4);
		\draw (u_5) --  (u_1);
		\draw (u_5) --  (u_2);
		\draw (u_5) --  (u_3);
		\draw (u_5) --  (u_4);
		
		\node[enclosed] (v_1) at (3.5,0) {};
		\node[enclosed] (v_2) at (4.5,0.6) {};
		\node[enclosed] (v_3) at (5.5,0) {};
		\node[enclosed] (v_4) at (4.5,-0.7) {};
		\node[enclosed, fill=cyan] (v_5) at (3.9,1.9) {};
		\node[enclosed, fill=cyan] (v_6) at (5.1,1.9) {};
		
		\draw (v_1) --  (v_2);
		\draw (v_2) --  (v_3);
		\draw (v_3) --  (v_4);
		\draw (v_1) --  (v_4);
		\draw (v_5) --  (v_1);
		\draw (v_5) --  (v_2);
		\draw (v_5) --  (v_3);
		\draw (v_5) --  (v_4);
		\draw (v_6) --  (v_1);
		\draw (v_6) --  (v_2);
		\draw (v_6) --  (v_3);
		\draw (v_6) --  (v_4);
		\draw 
		[dashed] (v_6) --  (v_5);
		
		\node[enclosed] (w_1) at (7.5,0) {};
		\node[enclosed] (w_2) at (8.5,-0.7) {};
		\node[enclosed] (w_3) at (9.5,0) {};
		\node[enclosed, fill=cyan] (w_4) at (7.9,1.9) {};
		\node[enclosed, fill=cyan] (w_5) at (9.1,1.9) {};
		
		\draw (w_1) --  (w_2);
		\draw (w_3) --  (w_2);
		\draw (w_4) --  (w_1);
		\draw (w_4) --  (w_2);
		\draw (w_4) --  (w_3);
		\draw (w_5) --  (w_1);
		\draw (w_5) --  (w_2);
		\draw (w_5) --  (w_3);
		\draw [dashed] (w_4) --  (w_5);
		
		\end{tikzpicture}
	\end{center}
	\caption{A cone on $K_4$ (left), a double cone on $C_4$ (center) and a double cone on $P_3$ (right)}\label{DC}
\end{figure}

\begin{corollary}
\label{dcsimp}
With the assumption in Theorem \ref{dc2}(\ref{yes}), we further suppose that $\operatorname{deg}(w)-(A)_{w,w}$ is a constant $d$ for all $w\in V(H)$. 
\begin{enumerate}
\item If $\eta=0$, then the apexes of $X\vee H$ are strongly cospectral with respect to $\textbf{A}(X\vee H)$.
\item \label{dcsimpweight} Let $\eta\neq 0$ and $X\vee H$ be simple. The apexes of $X\vee H$ are strongly cospectral with respect to $\textbf{A}(X\vee H)$ if and only if (\ref{scdoubcone2}) does not hold. In particular, if $\beta\neq -\gamma$ and $X\vee H$ is simple and unweighted, then the apexes of $X\vee H$ are strongly cospectral with respect to $\textbf{A}(X\vee H)$ if and only if $d\neq m-1$.
\end{enumerate}
\end{corollary}

%The condition $d\neq m-1$ in Corollary \ref{dcsimp}(2) is equivalent to  $H\not\cong K_m$. 

Thus, the apexes of a simple connected double cone $X\vee H$ in Theorem \ref{dcsimp}(\ref{dcsimpweight}) are adjacency strongly cospectral if and only if $d\neq \frac{1}{\eta}\left(\delta^2m-\eta^2\right)$, while they are signless Laplacian strongly cospectral if and only if $d\neq \frac{1}{2}\delta\left[(m-2)+\frac{\delta m}{\eta}\right]$. Lastly, with the assumption in Theorem \ref{dc2}(\ref{yes}), we further let $\beta=-\gamma$. Then $\pi$ is almost equitable, and we can write (\ref{M}) as
\begin{equation}
\label{scdoubcone1}
\eta\left[2\eta+(m-2)\delta+\omega-\frac{1}{m}\sum_{w\in H}(A)_{w,w}\right]-\delta^2m=0,
\end{equation}
which yields another corollary to Theorem \ref{dc2}(\ref{yes}).

\begin{corollary}
\label{dcsimp1}
With the assumption in Theorem \ref{dc2}(\ref{yes}), we further let $\beta=-\gamma$.
\begin{enumerate}
\item If $\eta=0$, then the apexes of $X\vee H$ are strongly cospectral with respect to $\textbf{A}(X\vee H)$.
\item \label{dcsimpweight1} Assume $\eta\neq 0$. Then the apexes of $X\vee H$ are strongly cospectral with respect to $\textbf{A}(X\vee H)$ if and only if (\ref{scdoubcone1}) does not hold. In particular, if $X\vee H$ is simple, the apexes of $X\vee H$ are strongly cospectral with respect to $\textbf{A}(X\vee H)$ if and only if $\eta\neq \delta$ and $m\neq -\frac{2\eta}{\delta}$.
\end{enumerate}
\end{corollary}

Combining Corollaries \ref{dcsimp} and \ref{dcsimp1} yields the following result about simple unweighted double cones.

\begin{corollary}
\label{cor}
Let $X\vee H$ be a simple unweighted double cone on a graph $H$ with $m$ vertices.
\begin{enumerate}
\item \label{cor1} Assume $H$ is $d$-regular. If $X\vee H$ is a disconnected double cone, then its apexes are strongly cospectral with respect to $\textbf{A}(X\vee H)$. If $X\vee H$ is a connected double cone and $\beta\neq -\gamma$, then the apexes of $X\vee H$ are strongly cospectral with respect to $\textbf{A}(X\vee H)$ if and only if $d\neq m-1$ (i.e., $H\not\cong K_m$).
\item \label{cor3} Let $\beta=-\gamma$. If $X\vee H$ is a disconnected double cone, then its apexes are strongly cospectral with respect to $\textbf{A}(X\vee H)$. If $X\vee H$ is a connected double cone, then its apexes are not strongly cospectral with respect to $\textbf{A}(X\vee H)$.
\end{enumerate}
\end{corollary}

The cone on $K_4$ does not exhibit adjacency, Laplacian, and signless Laplacian strong cospectrality by Corollary \ref{conesimpunw} while the apexes of the simple unweighted connected and disconnected double cones on $C_4$ are adjacency and signless Laplacian strongly cospectral by Corollary \ref{cor}(\ref{cor1}) (see Figure \ref{DC}). Moreover, Corollary \ref{cor}(\ref{cor3}) implies that the apexes of the simple unweighted disconnected double cone on $P_3$ (see Figure \ref{DC}) are Laplacian strongly cospectral, which is not the case if the apexes of this double cone are connected.

%%%%%%%%%%%%%%%%%%%%%%%%%%%%%%%%%%%%%%%

{\bf Acknowledgment.} H.M.\ is supported by the University of Manitoba's Faculty of Science and Faculty of Graduate Studies. Some of the work in this paper was completed as part of the author's MSc. thesis at the University of Manitoba. H.M.\ would like to thank the referees for the careful review and useful suggestions. H.M.\ would also like to thank Steve Kirkland and Sarah Plosker for the research guidance. H.M.\ also acknowledges the helpful comments from Chris Godsil and Gabriel Coutinho.

%%%%%%%%%%%%%%%%%%%%%%%%%%%%%%%%%%%%%%%

\end{document}